\documentclass[aos,preprint]{imsart}

\pdfoutput=1
\usepackage{amsthm,amsmath,natbib}
\usepackage[hypertexnames=false,bookmarksnumbered]{hyperref}
\usepackage{bookmark}
\usepackage{graphicx}
\usepackage{amssymb}
\usepackage{dsfont}
\usepackage{accents}
\usepackage{wrapfig}

\arxiv{arXiv:1607.02833}

\startlocaldefs

\newcommand{\R}{\ensuremath{\mathbb{R}}}
\newcommand{\Hyp}{\ensuremath{\mathbb{H}}}

\newcommand{\M}{\ensuremath{{\cal M}}}
\newcommand{\C}{\ensuremath{{\cal C}}}
\newcommand{\Pk}{\ensuremath{{\cal P}^*_k}}
\newcommand{\Mstar}{\ensuremath{{\cal M}^*{(x_0, \ldots x_k)}}}

\DeclareMathOperator{\Aff}{Aff}
\DeclareMathOperator{\Span}{Span}
\DeclareMathOperator{\EBS}{EBS}
\DeclareMathOperator{\FBS}{FBS}
\DeclareMathOperator{\KBS}{KBS}
\DeclareMathOperator{\arccosh}{arccosh}

\newcommand{\dist}{\ensuremath{\:\mbox{\rm dist}}}
\newcommand{\scal}[2]{\ensuremath{ \left<
    \:#1\:\left|\:#2\right.\right> }}
\newcommand{\inv}{^{\text{\tiny (-1)}}}

\newcommand{\trp}{^{\text{\tiny T}}} 

\newcommand{\Id}{\ensuremath{\:\mathrm{Id}}}
\newcommand{\one}{\ensuremath{\mathds{1}}}
\newcommand{\tr}{\mbox{\rm Tr}}
\newcommand{\lcp}[1]{\ensuremath{\overrightarrow{#1}}}

\newcommand\ubar[1]{\underaccent{\bar}{#1}}
\newcommand{\nlambda}{\ubar{\lambda}}

\newtheorem{theorem}{Theorem}
\newtheorem{definition}{Definition}
\newtheorem{proposition}{Proposition}

\endlocaldefs

\begin{document}

\begin{frontmatter}

\pdfbookmark[0]{Barycentric Subspace Analysis on Manifolds}{title}
\title{Barycentric Subspace Analysis on Manifolds}
\runtitle{Barycentric Subspace Analysis on Manifolds}


\author{\fnms{Xavier} \snm{Pennec}\corref{}\ead[label=e1]{xavier.pennec@inria.fr}}

\affiliation{Universit\'e C\^ote d'Azur and Inria, France}

\address{Asclepios team, Inria Sophia-Antipolis M\'editerrann\'ee\\ 2004 Route des Lucioles, BP93
\\ F-06902 Sophia-Antipolis Cedex, France\\ \printead{e1}}

\runauthor{X. Pennec}

\begin{abstract}
This paper investigates the generalization of Principal Component Analysis (PCA) to Riemannian manifolds. We first propose a new and  general type of family of subspaces in manifolds that we call barycentric subspaces. They are implicitly defined as the locus of points which are weighted means of $k+1$ reference points.  As this definition relies on points and not on tangent vectors,  it can also be extended to geodesic spaces which are not Riemannian. For instance, in stratified spaces, it naturally allows principal subspaces that span several strata, which is impossible in previous generalizations of PCA. We show that barycentric subspaces locally define a submanifold of dimension $k$ which generalizes geodesic subspaces.

Second, we rephrase PCA in Euclidean spaces as an optimization on flags of linear subspaces (a hierarchy of properly embedded linear subspaces of increasing dimension). We show that the Euclidean PCA minimizes the Accumulated Unexplained Variances by all the subspaces of the flag (AUV).  Barycentric subspaces are naturally nested, allowing the  construction of hierarchically nested subspaces. Optimizing the AUV criterion to optimally approximate data points with flags of affine spans in Riemannian manifolds  lead to a particularly appealing generalization of PCA on manifolds called Barycentric Subspaces Analysis (BSA). 
\end{abstract}



\end{frontmatter}

\section{Introduction}

In a Euclidean space, the principal $k$-dimensional affine subspace of the Principal Component Analysis (PCA) procedure is equivalently defined by minimizing the variance of the residuals (the projection of the data point to the subspace) or by maximizing the explained variance within that affine subspace. This double interpretation is available through Pythagoras' theorem, which does not hold in more general manifolds. A second important observation is that principal components of different orders are nested, enabling the forward or backward construction of nested principal components. 

Generalizing PCA to manifolds first requires the definition of the equivalent of affine subspaces in manifolds. For the zero-dimensional subspace, an intrinsic generalization of the mean on manifolds naturally comes into mind: the Fr\'echet mean is the set of global minima of the variance, as defined by \cite{frechet48} in general metric spaces. For simply connected Riemannian manifolds of non-positive curvature, the minimum is unique and is  called the Riemannian center of mass. This fact was already known by Cartan in the 1920's, but was not used for statistical purposes. \cite{karcher77,buser_gromovs_1981} first established conditions on the support of the distribution to ensure the uniqueness of a local minimum in general Riemannian manifolds. This is now generally called Karcher mean, although there is a dispute on the naming \citep{karcher_riemannian_2014}. From a statistical point of view, \cite{Bhattacharya:2003,Bhattacharya:2005} have studied in depth the asymptotic properties of the empirical Fr\'echet / Karcher means.

The one-dimensional component can naturally be a geodesic passing through the mean point. Higher-order components are more difficult to define. The simplest generalization is tangent PCA (tPCA), which amounts unfolding the whole distribution in the tangent space at the mean, and computing the principal components of the covariance matrix in the tangent space. The method  is thus based on the maximization of the explained variance, which is consistent with the entropy maximization definition of a Gaussian on a manifold  proposed by \cite{pennec:inria-00614994}. tPCA is actually implicitly used in most statistical works on shape spaces and Riemannian manifolds because of its simplicity and efficiency. However, if tPCA is good for analyzing data which are sufficiently centered around a central value (unimodal or Gaussian-like data), it is often not sufficient for distributions which are multimodal or supported on large compact subspaces (e.g. circles or spheres). 

Instead of an analysis of the covariance matrix, \cite{fletcher_principal_2004} proposed the minimization of squared distances to subspaces which are totally geodesic at a point, a procedure coined Principal Geodesic Analysis (PGA). These  Geodesic Subspaces (GS) are spanned by the geodesics going through a point with tangent vector restricted to a linear subspace of the tangent space. However, the least-squares procedure is computationally expensive, so that the authors approximated it in practice with tPCA, which led to confusions between tPCA and PGA. A real implementation of the  original PGA procedure was only recently provided  by \cite{sommer_optimization_2013}. PGA is allowing to build a flag (sequences of embedded subspaces) of principal geodesic subspaces consistent with a forward component analysis approach. Components are built iteratively from the mean point by selecting the tangent direction that optimally reduces the square distance of data points to the geodesic subspace. In this procedure, the mean  always belongs to geodesic subspaces even when it is outside of the distribution support.

To alleviate this problem, \cite{huckemann_principal_2006}, and later \cite{huckemann_intrinsic_2010}, proposed to start at the first order component directly with the geodesic best fitting the data, which is not necessarily going through the mean. The second principal geodesic is chosen orthogonally to the first one, and  higher order components are added orthogonally at the crossing point of the first two components. The method was named  Geodesic PCA (GPCA). Further relaxing the assumption that second and higher order components should cross at a single point, \cite{sommer_horizontal_2013} proposed a parallel transport of the second direction along the first principal geodesic to define the second coordinates, and iteratively define higher order coordinates through horizontal development along the previous modes.   

These are all intrinsically forward methods that build successively larger approximation spaces for the data. A notable exception is the concept of Principal Nested Spheres (PNS), proposed by \cite{jung_analysis_2012} in the context of planar landmarks shape spaces. A backward analysis approach determines a decreasing family of nested subspheres by slicing a higher dimensional sphere with affine hyperplanes. In this process, the nested subspheres are not of radius one, unless the hyperplanes passe through the origin. \cite{damon_backwards_2013} have recently generalized this approach to manifolds with the help of a ``nested sequence of relations''. However, up to now, such a sequence was only known for spheres or Euclidean spaces. 

We first propose in this paper new types of family of subspaces in manifolds: barycentric subspaces generalize geodesic subspaces  and  can naturally be nested,  allowing the construction of inductive forward or backward nested subspaces. We then rephrase PCA in Euclidean spaces as an optimization on flags of linear subspaces (a hierarchy of properly embedded linear subspaces of increasing dimension). To that end, we propose an extension of the unexplained variance criterion that generalizes nicely to flags of barycentric subspaces in Riemannian manifolds. This leads to a particularly appealing generalization of PCA on manifolds: Barycentric Subspaces Analysis (BSA).

\subsection*{Paper Organization}
We recall in Section \ref{Sec:Geom} the notions and notations needed to define statistics on Riemannian manifolds, and we introduce the two running example manifolds of this paper: $n$-dimensional spheres and hyperbolic spaces.  Exponential Barycentric Subspaces (EBS) are then defined in Section \ref{Sec:Bary} as the locus of weighted exponential barycenters of $k+1$ affinely independent reference points.  The  closure of the EBS in the original manifold is called affine span (this differs from the preliminary definition of \cite{pennec:hal-01164463}).  Equations of the EBS and affine span are exemplified on our running examples: the affine span  of $k+1$ affinely independent reference points is the great subsphere (resp.  sub-hyperbola) that contains the reference points. In fact, other tuple of points of that subspace generates the same affine span, which is also a geodesic subspace. This coincidence is due to the very high symmetry of the constant curvature spaces. 

Section \ref{Sec:KBS} defines the  Karcher (resp. Fr\'echet)  barycentric subspaces (KBS, resp. FBS)  as the  local (resp. global) minima of the weighted squared distance to the reference points. As the definitions relies on distances between points and not on tangent vectors, they are also valid in more general non-Riemannian geodesic spaces. For instance, in stratified spaces,  barycentric subspaces may naturally span  several strata. For Riemannian manifolds, we show that our three definitions are subsets of each other (except possibly at the cut locus of the reference points): the largest one, the EBS, is composed of the critical points of the weighted variance. It forms a cell complex according to the index of the critical points. Cells of positive index gather local minima to form the KBS. 
We explicitly compute the Hessian on our running spherical and hyperbolic examples.  Numerical tests show that the index can be arbitrary, thus subdividing the EBS into several regions for both positively and negatively curved spaces. Thus, the KBS consistently covers only a small portion of the affine span in general and is a less interesting definition for subspace analysis purposes.

For affinely independent points, we show in Section \ref{Sec:Prop} that the regular part of a barycentric subspace is a stratified space which is locally a submanifold of dimension $k$. At the limit, points may coalesce along certain directions, defining non local jets\footnote{$p$-jets are equivalent classes of functions up to order $p$. Thus, a $p$-jet specifies the Taylor expansion of a smooth function up to order $p$. Non-local jets, or multijets, generalize subspaces of the tangent spaces to higher differential orders  with multiple base points.} instead of a regular $k+1$-tuple. Restricted geodesic subspaces, which are defined by $k$ vectors tangent at a point, correspond  to the limit of the affine span when the $k$-tuple converges towards that jet. 

Finally, we discuss in Section \ref{Sec:BSA} the use of these barycentric subspaces to generalize PCA on manifolds. Barycentric subspaces can be naturally nested by defining an ordering of the reference points. Like for PGA, this enables the construction of a forward nested sequence of subspaces which contains the Fr\'echet mean. In addition, BSA also provides  backward nested sequences which may not contain the mean. However, the criterion on which these constructions are based can be optimized for each subspace independently but not consistently for the whole sequence of subspaces. In order to obtain a global criterion, we rephrase PCA in Euclidean spaces as an optimization on flags of linear subspaces (a hierarchies of properly embedded linear subspaces of increasing dimension). To that end, we propose an extension of the unexplained variance criterion (the Accumulated Unexplained Variance (AUV) criterion) that generalizes nicely to flags of affine spans in Riemannian manifolds. This results into a particularly appealing generalization of PCA on manifolds, that we call Barycentric Subspaces Analysis (BSA).

\section{Riemannian geometry}\label{Sec:Geom}

In Statistics, directional data occupy a place of choice \citep{dryden2005,huckemann_principal_2006}. Hyperbolic spaces are also the simplest models of negatively curved spaces which model the space of isotropic Gaussian parameters with the Fisher-Rao metric in information geometry \citep{costa_fisher_2015}. As non-flat constant curvature spaces, both spherical and hyperbolic spaces are now considered in manifold learning for embedding data \citep{wilson_spherical_2014}. Thus, they are ideal examples to illustrate the theory throughout this paper.

\subsection{Tools for computing on Riemannian manifolds}
We consider a differential manifold  $\M$ endowed with a smooth scalar product $\scal{.}{.}_{x}$ called the Riemannian metric on each tangent space $T_{x}\M$ at point $x$ of $\M$. In a chart, the metric is specified by the dot product of the tangent vector to the coordinate curves: $g_{ij}(x) = \scal{\partial_i}{\partial_j}_x$. The Riemannian distance between two points is the infimum of the length of the curves joining these points. Geodesics, which are critical points of the energy functional, are  parametrized by arc-length in addition to optimizing the length. We assume here that the manifold is geodesically complete, i.e. that the definition domain of all geodesics can be extended to $\R$. This means that the manifold has no boundary nor any singular point that we can reach in a finite time. As an important consequence, the Hopf-Rinow-De~Rham theorem states that there always exists at least one minimizing geodesic between any two points of the manifold (i.e. whose length is the distance between the two points).

\paragraph{Normal coordinate system}
From the theory of second order differential equations, we know that there exists one and only one geodesic $\gamma_{(x,v)}(t)$ starting from the point $x$ with the tangent vector $v \in T_{x}\M$. The exponential map at point $x$  maps each tangent vector $v \in T_{x}\M$ to the point of the manifold that is reached after a unit time by the geodesic: $ \exp_{x}(v) = \gamma_{(x,v)}(1)$. The exponential map is locally one-to-one around $0$:  we denote by $\lcp{xy}=\log_{x}(y)$ its inverse.  The injectivity domain is the maximal domain $D(x) \subset T_{x}\M$ containing $0$ where the exponential map is a diffeomorphism. This is a connected star-shape domain limited by the tangential cut locus $\partial D(x) = C(x) \subset T_{x}\M$ (the set of vectors $t v$ where the geodesic $\gamma_{(x, v)}(t)$ ceases to be length minimizing). The cut locus $\C(x) = \exp_{x}(C(x)) \subset \M$ is the closure of the set of points where several minimizing geodesics starting from $x$ meet. The image of the domain $D(x)$ by the exponential map covers all the manifold except the cut locus, which has null measure. Provided with an orthonormal basis, exp and log maps realize a normal coordinate system at each point $x$. Such an atlas is the basis of programming on Riemannian manifolds as exemplified in \cite{pennec:inria-00614990}.

\paragraph{Hessian of the squared Riemannian distance}
On $\M \setminus C(y)$, the Riemannian gradient $\nabla^a = g^{ab} \partial_b$ of the squared distance $d^2_y(x)=\dist^2(x, y)$ with respect to the fixed point $y$ is $\nabla d^2_y(x) = -2 \log_x(y)$. The Hessian operator (or double covariant derivative) $\nabla^2$ is the covariant derivative of the gradient. In a normal coordinate at the point $x$, the Christoffel symbols vanish at $x$ so that the Hessian  of the square distance can be expressed with the standard differential $D_x$ with respect to the footpoint $x$: $\nabla^2 d^2_y(x) = -2 (D_x \log_x(y))$. It can also be written in terms of the differentials of the exponential map  as $ \nabla^2 d^2_y(x) = ( \left. D\exp_x \right|_{\lcp{xy}})^{-1} \left. D_x \exp_x \right|_{\lcp{xy}}$ to explicitly make the link with Jacobi fields. Following \cite{brewin_riemann_2009}, we computed in \ref{suppA} the Taylor expansion of this matrix in a normal coordinate system  at $x$:
\begin{equation}
\label{eq:Diff_log} 
\left[  D_x \log_x(y) \right]^a_b = -\delta^a_b + \frac{1}{3} R^a_{cbd} \lcp{xy}^c \lcp{xy}^d + \frac{1}{12} \nabla_c R^a_{dbe}  \lcp{xy}^c \lcp{xy}^d \lcp{xy}^e  +  O(\varepsilon^3).
\end{equation}
Here, $R^a_{cbd}(x)$ are the coefficients of the curvature tensor at $x$ and Einstein summation convention implicitly sums  upon each index that appear up and down in the formula. Since we are in a normal coordinate system,  the zeroth order term is the identity matrix, like in Euclidean spaces, and the first order term vanishes. The Riemannian curvature tensor appears in the second order term and its covariant derivative in the third order term. Curvature is the leading term that makes this matrix departing from the identity (the Euclidean case) and may lead to the non invertibility of the differential.

\paragraph{Moments of point distributions}

Let $\{x_0,\ldots x_k\}$ be a set of $k+1$ points on a Manifold provided with weights $(\lambda_0, \ldots \lambda_k)$ that do not sum to zero. We may see these weighted points as the weighted sum of  Diracs $\mu(x) = \sum_i \lambda_i \delta_{x_i}(x)$.  As this distribution is not normalized and weights can be negative, it is generally not a probability. It is also singular with respect to the Riemannian measure. Thus, we have to take care in defining its moments as the Riemannian log and distance functions are not smooth at the cut-locus.  

\begin{definition} [$(k+1)$-pointed / punctured Riemannian manifold] $ $ \\
Let $\{x_0, \ldots x_k\} \in \M^{k+1}$  be a set of $k+1$ reference points in the $n$-dimensional Riemannian manifold $\M$ and $C(x_0, \ldots x_k) =  \cup_{i=0}^k C(x_i)$ be the union of the cut loci of these points. We call  the object consisting of the smooth manifold $\M$ and the $k+1$ reference points a $(k+1)$-pointed manifold. Likewise, we call the submanifold $ \Mstar = \M \setminus C(x_0, \ldots x_k)$ of the non-cut points of the $k+1$ reference points a $(k+1)$-punctured manifold.
\end{definition}
On $ \Mstar$, the distance to the points  $\{x_0, \ldots x_k\}$ is smooth. The Riemannian log function $\lcp{x x_i} = \log_x(x_i)$ is also well defined for all the points of $\Mstar$. Since the cut locus of each point is closed and has null measure, the punctured manifold $\Mstar$ is open and dense in $\M$, which means that it is a submanifold of $\M$. However, this submanifold is not necessarily connected.  For instance in the flat torus $(S_1)^n$, the cut-locus of $k+1 \leq n$ points divides the torus into $k^n$ disconnected cells.

\begin{definition} [Weighted moments of a $(k+1)$-pointed manifold] $ $ \\
Let $(\lambda_0, \ldots \lambda_k) \in \R^{k+1}$ such that $\sum_i \lambda_i \not = 0$. We call $\nlambda_i = \lambda_i / (\sum_{j=0}^k \lambda_j)$ the normalized weights. The weighted $p$-th order moment of a $(k+1)$-pointed Riemannian manifold is the $p$-contravariant tensor:
\begin{equation}
 {\mathfrak M}_{p}(x,\lambda) = \sum_i \lambda_i \underbrace{\lcp{xx_i} \otimes \lcp{xx_i} \ldots  \otimes  \lcp{xx_i}}_{\text{$p$ times}}.
\end{equation}
The normalized $p$-th order moment is: $\underline{\mathfrak M}_p(x,\lambda) = {\mathfrak M}_p(x,\nlambda) = {\mathfrak M}_p(x,\lambda)/ {\mathfrak M}_0(\lambda).$ Both tensors are smoothly defined on the punctured manifold $\Mstar$.
\end{definition}
The 0-th order moment ${\mathfrak M}_0(\lambda) = \sum_i \lambda_i = \mathds{1}\trp \lambda$ is the mass. The $p$-th order moment is homogeneous of degree 1 in $\lambda$ while the normalized $p$-th order moment is naturally invariant by a change of scale of the weights. For a fixed weight $\lambda$, the first order moment ${\mathfrak M}_1(x,\lambda)  =  \sum_i \lambda_i \lcp{xx_i}$ is a smooth vector field on the manifold $\Mstar$ whose zeros will be the subject of our interest. The second and higher order moments are smooth $(p,0)$ tensor fields that will be used in contraction with the Riemannian curvature tensor.

\paragraph{Affinely independent points on a manifold}

In a Euclidean space, $k+1$ points are affinely independent if their affine combination generates a $k$ dimensional subspace, or equivalently if none of the point belong to the affine span of the $k$ others. They define in that case a $k$-simplex. Extending these different definitions to manifolds lead to different notions. We chose a definition which rules out the singularities of constant curvature spaces and which guaranties the existence of barycentric subspaces around reference point. In the sequel, we assume by default that the $k+1$ reference points of pointed manifolds are affinely independent (thus $k \leq n$). Except for a few examples, the study of singular configurations is left for a future work.

\begin{definition}[Affinely independent points]
A set of $k+1$ points $\{x_0,\ldots x_k\}$ is affinely independent if no point is in the cut-locus of another and if all the sets of $k$ vectors $\{ \log_{x_i}(x_j) \}_{0 \leq j \not = i \leq k} \in T_{x_i}\M^k$ are linearly independent.
\label{def:AffineIndependence}
\end{definition}

\subsection{Example on the sphere ${\cal S}_n$}

We consider the unit sphere in dimension $n \geq 1$ embedded in $\R^{n+1}$. The tangent space at $x$ is  the space of vectors orthogonal to $x$:  $T_x{\cal S}_n = \{ v \in \R^{n+1}, v\trp x =0\}$ and the Riemannian metric is inherited from the Euclidean metric of the embedding space. With these conventions, the Riemannian distance is the arc-length $d(x,y) = \arccos( x\trp y)= \theta \in [0,\pi]$. Using the smooth function $f(\theta) =  { \theta}/{\sin\theta}$ from $]-\pi;\pi[$ to $\R$ which is always greater than one, the spherical exp and log maps are:
\begin{eqnarray}
\exp_x(v) & = & \cos(\| v\|) x +  \sin(\| v\|) v / \| v\| \\
\log_x(y) & = & f(\theta) \left( y - \cos\theta\: x \right)
\quad \text{with} \quad \theta = \arccos(x\trp y).
\end{eqnarray}
\paragraph{Hessian} The orthogonal projection  $v=(\Id -x x\trp)w$ of a vector $w \in \R^{n+1}$ onto the tangent space $T_x{\cal S}_n$ provides a chart around a point $x\in {\cal S}_n$ where we can compute the gradient and Hessian of the squared Riemannian distance (detailed in \ref{suppA}).  Let $u={(\Id-xx\trp)y }/ {\sin \theta} = { \log_x(y) }/{\theta}$ be the unit tangent vector pointing from $x$ to $y$, we obtain:
\begin{eqnarray}
 H_x(y) = \nabla^2 d^2_y(x) & =  &2 u u\trp + 2 f( \theta  )\cos\theta  (\Id -xx\trp - u u\trp).
\label{eq:HessDistSphere2}
\end{eqnarray}
By construction, $x$ is an eigenvector with eigenvalue $0$. Then the vector $u$ (or equivalently $\log_x(y) = \theta u$) is an eigenvector with eigenvalue $1$. To finish, every vector which is orthogonal to these two vectors (i.e. orthogonal to the plane spanned by 0, $x$ and $y$) has eigenvalue $ f(\theta)\cos\theta = \theta \cot \theta$. This last eigenvalue is positive for $\theta \in [0,\pi/2[$, vanishes for $\theta = \pi/2$ and becomes negative for $\theta \in ]\pi/2 \pi[$. We retrieve here the results of \cite[lemma 2]{buss_spherical_2001} expressed in a more general coordinate system.

\paragraph{Moments of a $k+1$-pointed sphere}
 We  denote a set of $k+1$ point on the sphere and the matrix of their coordinates by $X=[x_0,\ldots x_k]$. The cut locus of $x_i$ is its antipodal point $-x_i$ so that the $(k+1)$-punctured manifold is $\Mstar = {\cal S}_n \setminus -X$. Using the invertible diagonal matrix $F(X,x) = \mbox{Diag}( f( \arccos(x_i \trp x) ) )$, the first weighted moment is:
\begin{equation}
{\mathfrak M}_1(x, \lambda)  = \textstyle \sum_i \lambda_i \lcp{x x_i} = ( \Id -  x x\trp) X F(X,x) \lambda.
\label{eq:MomemtSphere}
\end{equation}

\paragraph{Affine independence of the reference points}
Because no point is antipodal nor identical to another, the plane generated by 0, $x_i$ and $x_j$ in the embedding space is also generated by 0, $x_i$ and the tangent vector $\log_{x_i}(x_j)$. This can be be seen using a stereographic projection of pole $-x_i$ from ${\cal S}_n$ to $T_{x_i} {\cal S}_n$. Thus, 0, $x_i$ and the $k$ independent vectors $\log_{x_i}(x_j)$ ($j \not = i$)  generate the same linear subspace of dimension $k+1$ in the embedding space than the points $\{0, x_0,\ldots x_k\}$. We  conclude that $k+1$ points on the sphere are affinely independent if and only if the matrix $X=[x_0,\ldots x_k]$ has  rank $k+1$.

\subsection{Example on the hyperbolic space $\Hyp^n$}
\label{sec:hyperbolic}

We now consider the hyperboloid of equation $-x_0^2 + x_1^2 \ldots  x_n^2 = -1$ ($x_0 > 0$)  embedded in $\R^{n+1}$ ($n \geq 2$). Using the notation $x=(x_0,\hat x)$ and the indefinite non-degenerate symmetric bilinear form $\scal{x}{y}_* = x\trp J y= \hat x\trp \hat y -x_0 y_0$ with $ J = \mbox{diag}(-1, \Id_n)$, the hyperbolic space $\Hyp^n$ can be seen as the pseudo-sphere $\|x\|^2_* = \|\hat x\|^2 -x_0^2 = -1$ of radius -1 in the Minkowski space $\R^{1,n}$. A point can be parametrized  by $x=(\sqrt{1+\|\hat x\|^2}, \hat x)$ for $\hat x \in \R^n$ (Weierstrass coordinates). The restriction of the Minkowski pseudo-metric of the embedding space $\R^{1,n}$ to the tangent space of $T_x\Hyp^n$ is positive definite. It defines the natural Riemannian metric on the hyperbolic space. With these conventions, geodesics are the trace of 2-planes passing through the origin and the Riemannian distance is the arc-length $d(x,y) = \arccosh( - \scal{x}{y}_* )$. Using the smooth positive function $f_*(\theta) = { \theta}/{\sinh(\theta)}$ from $\R$ to $]0,1]$, the hyperbolic exp and log maps are:
\begin{eqnarray}
\exp_x(v) & = & \cosh(\| v\|_* ) x +  {\sinh(\| v\|_* )} v / {\| v\|_* } \\
\log_x(y) & = & f_*(\theta) \left( y - \cosh(\theta) x \right)
\quad \text{with} \quad \theta = \arccosh( -\scal{x}{y}_* ).
\end{eqnarray}

\paragraph{Hessian} The orthogonal projection  $v=w + \scal{w}{x}_* x = (\Id + x x\trp J) w$ of a  vector $w\in \R^{1,n}$ onto the tangent space at $T_x \Hyp^n$ provides a chart around the point $x\in \Hyp^n$ where we can compute the gradient and Hessian of the hyperbolic squared distance (detailed in \ref{suppA}).  Let $u= { \log_x(y) }/{\theta}$ be the unit tangent vector pointing from $x$ to $y$, the Hessian is:
\begin{equation}
 H_x(y) = \nabla^2 d^2_y(x) = 2  J \left( u u\trp + \theta \coth \theta (J + x x\trp -u u\trp) \right) J  
\label{eq:GradDistHyper}
\end{equation}
By construction, $x$ is an eigenvector with eigenvalue $0$. The vector $u$ (or equivalently $\log_x(y) = \theta u$) is an eigenvector with eigenvalue $1$. Every vector orthogonal to these two vectors (i.e. to the plane spanned by 0, $x$ and $y$) has eigenvalue $ \theta \coth \theta \geq 1$ (with equality only for $\theta=0$). Thus, the Hessian of the squared distance is always positive definite. As a consequence, the squared distance is a convex function and has a unique minimum. This was of course expected for a negatively curved space \citep{bishop_manifolds_1969}.

\paragraph{Moments of a $k+1$-pointed hyperboloid}
We now pick $k+1$ points on the hyperboloid whose matrix of coordinates is denoted by $X=[x_0,\ldots x_k]$. Since there is no cut-locus, the $(k+1)$-punctured manifold is the manifold itself: $\Mstar = \M = {\Hyp}^n$. Using the invertible diagonal matrix $F_*(X,x) = \mbox{Diag}( f_*( \arccosh( -\scal{x_i}{x}_* ) ) )$, the  first weighted moment is
\begin{equation}
{\mathfrak M}_1(x, \lambda)  = \textstyle \sum_i \lambda_i \log_x(x_i) = (\Id + x x\trp J) X F_*(X,x) \lambda.
\label{eq:MomemtHyperboloid}
\end{equation}

\paragraph{Affine independence} 
As for the sphere, the origin, the point $x_i$ and the $k$ independent vectors $\log_{x_i}(x_j) \in T_{x_i}{\Hyp}^n$ ($j \not = i$) generate the same $k+1$ dimensional linear subspace of the embedding Minkowski space $\R^{1,n}$ than the points $\{x_0, \ldots x_k\}$. Thus, $k+1$ points on the hyperboloid are affinely independent if and only if the matrix $X$ has rank $k+1$.

\section{Exponential Barycentric Subspaces (EBS) and Affine Spans}
\label{Sec:Bary}

\subsection{Affine subspaces in a Euclidean space}

In Euclidean PCA, a zero dimensional space is a point, a one-dimensional space is a line, and an affine subspace of dimension $k$ is generated by a point and $k \leq n$ linearly independent vectors. We can also generate such a subspace by taking the affine hull of $k+1$ affinely independent points: $\Aff(x_0,\ldots x_k) =\left\{ x = \sum_i \lambda_i x_i, \text{with} \sum_{i=0}^k \lambda_i = 1\right\}.$  These two definitions are equivalent in a Euclidean space, but turn out to have different generalizations in manifolds. 

When  there exists a vector of coefficients $\lambda = (\lambda_0, \lambda_1, \ldots, \lambda_k) \in \R^{k+1}$ (which do not sum to zero) such that $\sum_{i=0}^k \lambda_i (x_i-x) =0,$ then $\lambda$ is called the barycentric coordinates of the point $x$ with respect to the $k$-simplex $\{x_0, \ldots x_k\}$. When points are dependent, some extra care has to be taken to show that the affine span is still well defined but with a lower dimensionality.  Barycentric coordinates are homogeneous of degree one:
\begin{definition}[Projective space of barycentric coordinates (weights)] \label{def:Pk}
Barycentric  coordinates of $k+1$ points live in the real projective space $\R P^n =  (\R^{k+1} \setminus \{0\})/\R^*$ from which we remove the codimension 1 subspace $\mathds{1}^{\perp}$ orthogonal to the point $\mathds{1} = (1:1: \ldots 1)$:
\[ \textstyle
\Pk = \left\{   \lambda=(\lambda_0 : \lambda_1 : \ldots : \lambda_k) \in \R P^n \text{ s.t. } \mathds{1}^{\top}\lambda \not = 0   \right\}.
\]
\end{definition}

\begin{wrapfigure}[12]{r}{0.5\textwidth} 
\centering
\includegraphics[width=0.5\textwidth]{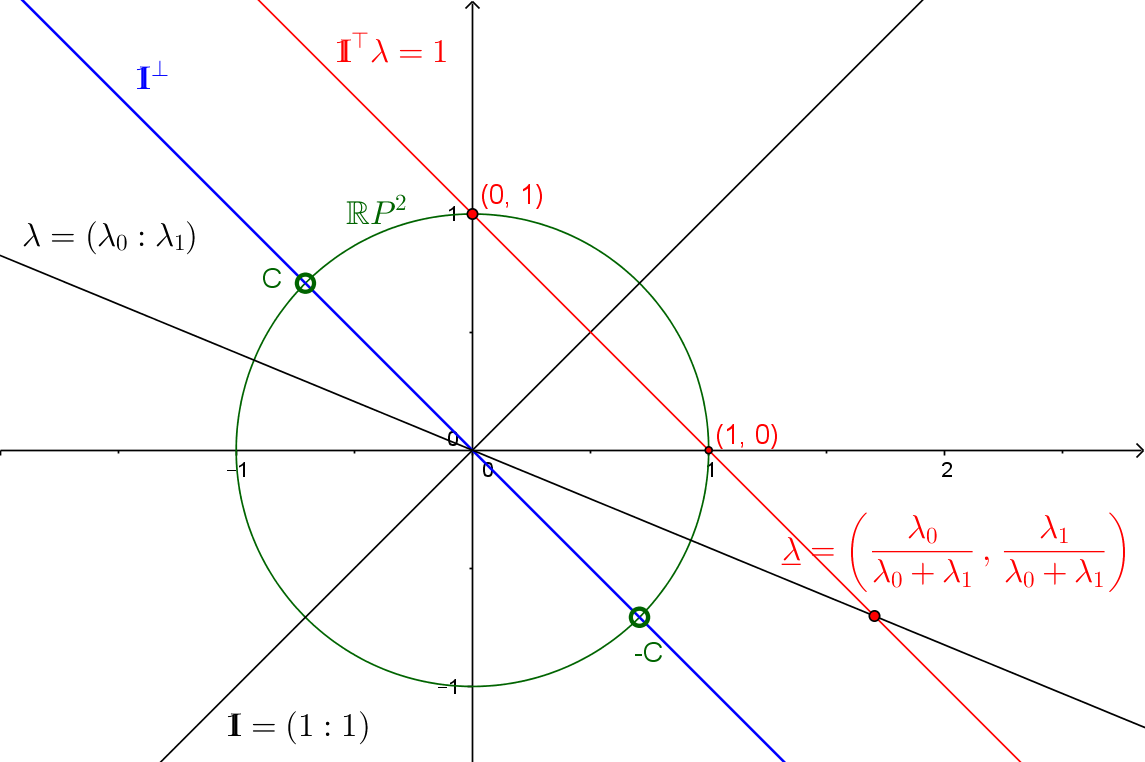}
\caption{Projective weights for $k=1$.}
\label{fig:P2}
\end{wrapfigure}
Projective points are represented by lines through 0 in Fig.\ref{fig:P2}.
Standard representations are given by the intersection of the lines with the "upper" unit sphere $S_k$ of $\R^{k+1}$  with north pole $\mathds{1}/\sqrt{k+1}$ or by the affine $k$-plane of $\R^{k+1}$ passing through  the point $\mathds{1}/(k+1)$ and orthogonal to this vector. This last representation give the normalized weight $ \underline{\lambda}_i= \lambda_i / (\sum_{j=0}^k \lambda_j)$: the vertices of the simplex have homogeneous coordinates $(1 : 0 : ... : 0) \ldots (0 : 0 : ... : 1)$. To prevent weights to sum up to zero, we have to remove the codimension 1 subspace $\mathds{1}^{\perp}$ orthogonal to the projective point $\mathds{1} = (1:1: \ldots 1)$ (blue line in Fig.\ref{fig:P2}). This excluded subspace corresponds to the equator of the pole  $\mathds{1}/\sqrt{k+1}$  for the sphere representation (points $C$ and $-C$ identified in Fig.\ref{fig:P2}), and to the projective completion (points at infinity) of the affine $k$-plane of normalized weights.

\subsection{EBS and Affine Span in Riemannian manifolds}

\begin{definition} [Barycentric coordinates in a $(k+1)$-pointed manifold] 
 A point $x \in \Mstar$ has barycentric coordinates $\lambda \in \Pk$ with respect to $k+1$ reference affinely independent points if  
\begin{equation}
\label{eq:Bary}
 {\mathfrak M}_1(x,\lambda) = \textstyle \sum_{i=0}^k \lambda_i \lcp{x x_i} = 0 . 
\end{equation}
\end{definition}
Since the Riemannian log function $\lcp{x x_i} = \log_x(x_i)$ is multiply defined on the cut locus of $x_i$, this definition cannot be extended to the the union of all cut loci $C(x_0, \ldots x_k)$, which is why we restrict the definition to $\Mstar$. 
\begin{definition}[Exponential Barycentric Subspace (EBS)] The EBS of the affinely independent points $(x_0,\ldots x_k) \in \M^{k+1}$ is the locus of weighted exponential barycenters of the reference points in $\Mstar$:
\[
\mbox{EBS}(x_0, \ldots x_k) = \{ x\in \Mstar | \exists \lambda \in \Pk :  {\mathfrak M}_1(x,\lambda) =0 \}.
\]
\end{definition}
The reference points could be seen as landmarks in the manifold.
This definition is fully symmetric wit respect to all of them, while one point is privileged in geodesic subspaces.  We could draw a link with archetypal analysis \citep{Cutler:1994:AA} which searches for extreme data values such that all of the data can be well represented as convex mixtures of the archetypes. However, extremality is not mandatory in our framework. 

\begin{proposition}[Dual subspace of barycentric coordinates] 
The subspace of barycentric coordinates $\Lambda(x) = \{ \lambda \in \Pk | {\mathfrak M}_1(x,\lambda) =0 \}$  at point $x \in \Mstar$ is either void, a point, or a linear subspace of $\Pk$.
\end{proposition}
We  see that a point belongs to $\EBS(x_0, \ldots x_k)$  if and only if  $\Lambda(x) \not = \emptyset$. Moreover, any linear combination of  weights that satisfy the equation is also a valid weight so that $\Lambda(x)$ can only be a unique point (dimension 0) or a linear subspace of $\Pk$. The dimension of the dual space $\Lambda(x)$ is actually controlling the local dimension of the barycentric space, as we will see below.

The discontinuity of the Riemannian log on the cut locus of the reference points may hide the continuity or discontinuities of the exponential barycentric subspace. In order to ensure the completeness  and potentially reconnect different components, we  consider the closure of this set.
\begin{definition}[Affine span of $k+1$ affinely independent points] 
The affine span is the closure of the EBS in $\M$: $ \Aff(x_0, \ldots x_k) = \overline{\mbox{EBS}}(x_0, \ldots x_k).$ Because we assumed that $\M$ is geodesically complete, this is equivalent to the metric completion of the EBS.  
\end{definition}

\subsection{Characterizations of the EBS}

Let $Z(x)= [\lcp{x x_0},\ldots \lcp{x x_k}]$ be the smooth field of $n\times (k+1)$ matrices of vectors pointing from any point $x \in \Mstar$ to the reference points. We can rewrite the constraint $\sum_i \lambda_i \lcp{x x_i} =0$ in matrix form: ${\mathfrak M}_1(x,\lambda) = Z(x)\lambda =0,$ where  $\lambda$ is the $k+1$ vector of homogeneous coordinates $\lambda_i$.

\begin{theorem}[Characterization of the exponential barycentric subspace] 
\label{THM1}
Let $Z(x)=U(x)\: S(x)\: V(x)\trp$ be a singular decomposition of the $n\times (k+1)$ matrix fields $Z(x)= [\lcp{x x_0},\ldots \lcp{x x_k}]$ on $\Mstar$ with singular values $\{s_i(x)\}_{0\leq i \leq k}$ sorted in decreasing order. $\mbox{EBS}(x_0, \ldots x_k)$ is the zero level-set of the smallest singular value $s_{k+1}(x)$ and the dual subspace of valid barycentric weights is spanned by the right singular vectors corresponding to the $l$ vanishing singular values: $\Lambda(x) = \Span(v_{k-l}, \ldots v_{k})$ (it is void if $l=0$).
\end{theorem}
\begin{proof}
Since $U$ and $V$ are orthogonal matrices, $Z(x)\lambda=0$ if and only if at least one singular value (necessarily the smallest one $s_{k}$) is null, and $\lambda$ has to live in the corresponding right-singular space: $\Lambda(x) = Ker(Z(x))$. If we have only one zero singular value ($s_{k+1}=0$ and $s_k>0$), then $\lambda$ is proportional to $v_{k+1}$. If $l$ singular values vanish, then we have a higher dimensional linear subspace of solutions for $\lambda$. 
\end{proof}

\begin{theorem} 
\label{THM5}
Let $G(x)$ be the matrix expression of the Riemannian metric in a local coordinate system and $\Omega(x) = Z(x)\trp  G(x) Z(x)$ be the smooth $(k+1)\times (k+1)$ matrix field  on $\Mstar$ with  components $\Omega_{ij}(x) = \scal{  \lcp{x x_i} }{ \lcp{x x_j} }_x$ and  $\Sigma(x) =  {\mathfrak M}_2(x,\mathds{1} ) = \sum_{i=0}^k \lcp{x x_i} \: \lcp{x x_i}\trp =  Z(x) Z(x)\trp$be the (scaled) \ $n \times n$ covariance matrix field of the reference points. $\EBS(x_0, \ldots x_k)$  is the zero level-set of: $\det(\Omega(x))$, the minimal eigenvalue $\sigma_{k+1}^2$ of $\Omega(x)$, the $k+1$ eigenvalue (in decreasing order) of the covariance $\Sigma(x)$.
\end{theorem}

\begin{proof}
The constraint ${\mathfrak M}_1(x,\lambda)=0$ is satisfied if and only if:
\[
\| {\mathfrak M}_1(x,\lambda) \|^2_x =  \left\| { \textstyle  \sum_i  \lambda_i \lcp{x x_i}} \right\|^2_{x} = {\lambda\trp \Omega(x) \lambda} =0.
\]
As the function is homogeneous  in $\lambda$, we can restrict to unit vectors. Adding this constrains with a Lagrange multiplier to the cost function, we end-up with the  Lagrangian ${\cal L}(x, \lambda, \alpha) = \lambda\trp \Omega(x) \lambda +\alpha (\lambda\trp \lambda -1)$. The minimum with respect to $\lambda$ is obtained for the eigenvector $\mu_{k+1}(x)$ associated to the smallest eigenvalue $\sigma_{k+1}(x)$ of $\Omega(x)$ (assuming that eigenvalues are sorted in decreasing order) and we have $\|{\mathfrak M}_1(x, \mu_{k+1}(x))\|^2_2 = \sigma_{k+1}(x)$, which is null if and only if the minimal eigenvalue is zero. Thus, the barycentric subspace of $k+1$ points is the locus of rank deficient matrices $\Omega(x)$:
\[
\EBS(x_0, \ldots x_k) = \phi\inv(0) \quad \mbox{where} \quad \phi(x) = \det(\Omega(x)).
\]

One may want to relate the singular values of $Z(x)$ to the eigenvalues of $\Omega(x)$. The later are the square of the  singular values of $G(x)^{1/2}Z(x)$. However, the left multiplication by the square root of the metric (a  non singular but non orthogonal matrix) obviously changes the singular values in general except for vanishing ones: the (right) kernels of $G(x)^{1/2}Z(x)$ and $Z(x)$ are indeed the same. This shows that the EBS is an affine notion rather than a metric one, contrarily to the Fr\'echet / Karcher  barycentric subspace. 

To draw the link with the $n\times n$ covariance matrix of the reference points, let us notice first that the definition does not assumes that the coordinate system is orthonormal. Thus, the eigenvalues of the covariance matrix depend on the chosen coordinate system, unless they vanish. In fact, only the joint eigenvalues of $\Sigma(x)$ and $G(x)$ really make sense, which is why this decomposition is called the proper orthogonal decomposition (POD). Now, the singular values of $Z(x)=U(x) S(x) V(x)\trp$ are also the square root of the first $k+1$ eigenvalues of $\Sigma(x) = U(x) S^2(x) U(x)\trp$, the remaining $n-k-1$ eigenvalues being null. Similarly, the singular values of $G(x)^{1/2}Z(x)$ are the square root of the first $k+1$ joint eigenvalues of $\Sigma(x)$ and $G(x)$. Thus, our barycentric subspace may also be characterized as the zero level-set of the $k+1$ eigenvalue (sorted in decreasing order) of $\Sigma$, and this characterization is once again independent of the basis chosen. 
\end{proof}

\subsection{Spherical EBS and affine span}
\label{Sec:SphericalEBS}

From Eq.(\ref{eq:MomemtSphere}) we identify the matrix: $Z(x) =  ( \Id -  x x\trp) X F(X,x).$ Finding points $x$ and weights $\lambda$ such that $Z(x)\lambda=0$ is a classical matrix equation, except for the  scaling matrix $F(X,x)$ acting on homogeneous projective weights, which is non-stationary and non-linear in both $X$ and $x$. However, since $F(X,x) = \mbox{Diag}( \theta_i /\sin \theta_i )$ is an invertible diagonal matrix, we can introduce {\em renormalized weights} $\tilde{\lambda} = F(X,x) \lambda, $ which leaves us with the equation $ ( \Id -  x x\trp) X \tilde \lambda=0$. The solutions  under the constraint $\|x\|=1$ are given by $(x\trp X \tilde{\lambda} ) x = X \tilde{\lambda}$ or more explicitly $x = \pm X \tilde{\lambda} / \| X \tilde{\lambda}\|$ whenever $X \tilde{\lambda} \not = 0$. This condition is ensured if $Ker(X)=\{0\}$. Thus, when the reference points are linearly independent, the point $x \in {\cal M}^*(X) $ has to belong to the Euclidean span of the reference vectors. Notice that for each barycentric coordinate we have two two antipodal solution points. Conversely, any unit vector  $x = X\tilde \lambda$  of the Euclidean span of $X$ satisfies the equation $( \Id -  x x\trp) X \tilde \lambda = (1-\|x\|^2) X \tilde \lambda =0$, and is thus a point of the EBS provided that it is not at the cut-locus of one of the reference points. This shows that 
\begin{equation}
 \EBS(X) = \Span\{x_0, \ldots x_k\} \cap {\cal S}_n \setminus X. 
\end{equation}

Using the renormalization principle, we can orthogonalize the reference points: let $X=U S V\trp$ be a singular value decomposition of the matrix of reference vectors. All the singular values $s_i$ are positive since the reference vectors $x_i$ are assumed to be linearly independent.  Thus, $\mu = S V\trp \tilde{\lambda} = S V\trp  F(X,x) \lambda$ is an invertible change of coordinate, and we are left with solving $ ( \Id -  x x\trp) U\mu =0$. By definition of the singular value decomposition, the Euclidean spans of $X$ and $U$ are the same, so that $\EBS(U) = \Span\{x_0, \ldots x_k\} \cap {\cal S}_n \setminus -U$.  This shows that the exponential barycentric subspace generated by the original points $X=[x_0, \ldots x_k]$ and the orthogonalized points $U=[u_0, \ldots u_k]$ are the same, except at the cut locus of all these points, but with different barycentric coordinates. 

To obtain the affine span, we take the closure of the EBS, which incorporates the cut locus of the reference points: $\Aff(X) = \Span\{x_0, \ldots x_k\}  \cap {\cal S}_n$. Thus, for spherical data as for Euclidean data, the affine span only depend on the reference points through the point of the Grassmanian they define.

\begin{theorem}[Spherical affine span] 
 The affine span $\Aff(X)$ of $k+1$ linearly independent reference unit points $X=[x_0,  \ldots x_k]$ on the $n$-dimensional sphere ${\cal S}_n$ endowed with the canonical metric is the  great subsphere of dimension k that contains the reference points.  \label{THM7}
\end{theorem}

When the reference points are affinely dependent on the sphere, the  matrix $X$ has one or more (say $l$) vanishing singular values. Any weight $\tilde{\lambda} \in \mbox{Ker}(X)$ is a barycentric coordinate vector for any point $x$ of the pointed sphere since the equation $( \Id -  x x\trp) X \tilde \lambda =0$ is  verified. Thus, the EBS is ${\cal S}_n \setminus -X$ and the affine span is the full sphere. If we exclude the abnormal subspace of weights valid for all points, we find that $x$ should be in the  span of the non-zero left singular vectors of $X$, i.e. in the  subsphere of dimension of dimension $rank(X)-1$ generated the Euclidean span of the reference vectors. This can also be achieved by focusing of the locus of points where $Z(x)$ has two vanishing singular values. This more reasonable result suggests adapting the EBS and affine span definitions for singular point configurations.

Two points on a 2-sphere is an interesting example that can be explicitly worked out. When the points are not antipodal, the rank of $X=[x_0,x_1]$ is 2, and the generated affine span is the one-dimensional geodesic joining the two points. When the reference points are antipodal, say north and south poles, X becomes rank one and one easily sees that all points of the 2-sphere are on one geodesic joining the poles with opposite log directions to the poles. This solution  of the EBS definition correspond to the renormalized weight $\tilde \lambda = (1/2 : 1/2) \in Ker(X)$ of the kernel of $X$. However, looking at the locus of points with two vanishing singular values of $Z(x)$ leads to restrict to the north and south poles only, which is a more natural and expected result.

\subsection{Hyperbolic EBS and affine span}
\label{Sec:HyperbolicEBS}

The hyperbolic case closely follows the spherical one.  From Eq.\eqref{eq:MomemtHyperboloid}, we get the expression of the matrix $ Z(x) = ( \Id +  x x\trp J) X F_*(X,x)$. Solving for $Z(x)\lambda=0$ can be done as previously by solving $ ( \Id +  x x\trp J) X \tilde \lambda =0$ with the renormalized weights $\tilde{\lambda} = F_*(X,x) \lambda$. This equation  rewrites $<x| X \tilde{\lambda}>_* x = - X \tilde{\lambda}$, so that the solution has to be of the form $X \tilde \lambda =0$ or $x = \alpha  X \tilde{\lambda}$. When the points are affinely independent, the first form is excluded since $Ker(X)=0$. In order to satisfy the constraint $\|x\|^2_*=-1$ in the second form, we need to have $\alpha^2 = - \|X \tilde{\lambda}\|_*^{-2} >0$ and the first coordinate $[X \tilde{\lambda}]_0$ of $X \tilde{\lambda}$ has to be positive. This defines a cone  in the space of renormalized weights from which each line parametrizes a point $x =  \text{sgn}( [X \tilde{\lambda}]_0) X \tilde \lambda / \sqrt {\tiny -\|X \tilde{\lambda}\|_*^2}$ of the Hyperbolic EBS. Thus, $\Aff(X)$ is the $k$-dimensional hyperboloid generated by the intersection of the Euclidean span of the reference vectors with the hyperboloid $\Hyp^n$. Since it is complete, the completion does not add anything to the affine span:
\begin{equation}
 \Aff(X) = \EBS(X) = \Span\{x_0, \ldots x_k\} \cap \Hyp^n. 
\end{equation}
As for spheres, we see that the hyperbolic affine span  only depend on the reference points through the point of the Grassmanian they define.

\begin{theorem}[Hyperbolic affine span] 
 The affine span $\Aff(X) = \EBS(X)$ of $k+1$ affinely independent reference points $X=[x_0,  \ldots x_k]$ on the $n$-dimensional hyperboloid $\Hyp^n$ endowed with the canonical Minkowski pseudo-metric of the embedding space $\R^{1,n}$ is the hyperboloid of dimension $k$ generated by the intersection of the hyperboloid with the hyperplane containing  the reference points.  \label{thm:HyperbolicSpan}
\end{theorem}

When the matrix $X$ has one or more vanishing singular values (affine dependance), all the points of the hyperboloid are solutions corresponding to weights from $Ker(X)$. Excluding these abnormal solutions and looking at the locus of points where $Z(x)$ has two vanishing singular values, we find that $x$ should be in the  span of the non-zero left singular vectors of $X$, i.e. in the  subsphere of dimension of dimension $rank(X)-1$ generated the Euclidean span of the reference vectors.

\section{Fr\'echet / Karcher Barycentric subspaces}
\label{Sec:KBS}

The reformulation of the affine span  as the weighted mean of $k+1$ points also suggests a definition using the Fr\'echet or the Karcher mean, valid in general metric spaces.

\begin{definition}[Fr\'echet / Karcher barycentric subspaces of $k+1$ points] 
Let $(\M, \dist)$ be a metric space of dimension $n$ and $(x_0,\ldots x_k) \in \M^{k+1}$ be $k+1\leq n+1$ distinct reference points. The (normalized) weighted variance at point $x$ with  weight $\lambda \in \Pk$ is: $\sigma^2(x,\lambda) = \frac{1}{2}\sum_{i=0}^k \nlambda_i \dist^2(x, x_i) = \frac{1}{2}\sum_{i=0}^k \lambda_i \dist^2(x, x_i) / (\sum_{j=0}^k \lambda_j).$ The Fr\'echet barycentric subspace of these points  is the locus of weighted Fr\'echet means of these points, i.e. the set of absolute minima of the weighted variance:
\[
\FBS(x_0, \ldots x_k) = \left\{ \arg\min_{x\in \M} \sigma^2(x, \lambda), \:  \lambda \in \Pk \right\}
\]
The Karcher barycentric subspaces $\KBS(x_0, \ldots x_k)$ are defined similarly with local minima instead of global ones.
\end{definition}
 In stratified metric spaces, for instance, the barycentric subspace spanned by points belonging to different strata naturally maps over several strata. This is a significant improvement over geodesic subspaces used in PGA which can only be defined within a regular strata. In the sequel, we only deal with the KBS/FBS of affinely independent points in a Riemannian manifold.

\subsection{Link between the different barycentric subspaces} 

In order to analyze the relationship between the Fr\'echet, Karcher and Exponential barycentric subspaces, we follow the seminal work of \cite{karcher77}. First, the locus of local minima (i.e. Karcher mean) is a superset of the global minima (Fr\'echet mean). On the punctured manifold  $\Mstar$, the  squared distance $d^2_{x_i}(x) = \dist^2(x, x_i)$  is smooth and its gradient is $\nabla d^2_{x_i}(x)  = -2 \log_x(x_i)$. Thus, one recognizes that the EBS equation  $\sum_i \nlambda_i \log_x(x_i) =0$ (Eq.(\ref{eq:Bary})) defines  nothing else than the critical points of the weighted variance: 
\[
FBS\cap {\cal M}^* \subset KBS \cap {\cal M}^* \subset  Aff \cap {\cal M}^* = EBS.
\]
Among the critical points with a non-degenerate Hessian, local minima are characterized by a positive definite Hessian. When the Hessian is degenerate, we cannot conclude on the local minimality without going to higher order differentials. The goal of this section is to subdivide the EBS into a cell complex according to the index of the Hessian operator of the variance: 
\begin{equation}
\textstyle
H(x,\lambda) =  \nabla^2 \sigma^2(x,\lambda)  = - \sum_{i=0}^k \nlambda_i D_x \log_x(x_i).
\label{eq:Hessian}
\end{equation}

Plugging the  value of the Taylor expansion of the differential of the log of Eq.(\ref{eq:Diff_log}), we  obtain the Taylor expansion:
\begin{equation}
\label{eq:TaylorH}
\left[ H(x, \lambda) \right]^a_b 
 =  \delta^a_b - \frac{1}{3} R^a_{cbd}(x) {\mathfrak M}^{cd}_2(x,\nlambda)
- \frac{1}{12} \nabla_c R^a_{dbe}(x) {\mathfrak M}_3^{cde}(x,\nlambda) + O(\varepsilon^4).
\end{equation}
The key factor in this expression is the contraction of the Riemannian curvature with the weighted covariance tensor of the reference points. This contraction is an extension of the Ricci curvature tensor. Exactly as the Ricci curvature tensor encodes  how the volume of an isotropic geodesic ball in the manifold deviates from the volume of the standard ball in a Euclidean space (through its metric trace, the scalar curvature), the extended Ricci curvature encodes how the volume of the geodesic ellipsoid $\lcp{xy}\trp {\mathfrak M}_2(x,\nlambda)\inv \lcp{xy} \leq \varepsilon $  deviates from the  volume of the standard Euclidean ellipsoid.

In locally symmetric affine spaces, the covariant derivative of the curvature is identically zero, which simplifies the formula. In the limit of null curvature, (e.g. for a locally Euclidean space like the torus), the Hessian matrix $H(x, \lambda)$ converges to the unit matrix and never vanishes. In general Riemannian manifolds, Eq.(\ref{eq:TaylorH}) only gives a qualitative behavior but does not provide guaranties as it is a series involving higher order moments of the reference points. In order to obtain hard bounds on the spectrum of $H(x, \lambda)$, one has to investigate bounds on Jacobi fields using Riemannian comparison theorems, as for the proof of uniqueness of the Karcher and Fr\'echet means (see \cite{karcher77,kendall90,Le:2004,Afsari:2010,Yang:2011}).

\begin{definition}[Degenerate, non-degenerate and positive points]
\label{def:NonDegenerate}
An exponential barycenter $x \in \EBS(x_0,\ldots x_k)$ is degenerate (resp. non-degenerate or positive) if the Hessian matrix $H(x,\lambda)$ is singular (resp. definite or positive definite) for all $\lambda$ in the the dual space of barycentric coordinates $\Lambda(x)$. The set of degenerate exponential barycenters is denoted by $EBS^0(x_0,\ldots,x_k)$ (resp. non-degenerate by $EBS^*(x_0,\ldots,x_k)$  and positive by $EBS^+(x_0,\ldots x_k)$). 
\end{definition}

 The definition of non-degenerate and positive points could be generalized to non-critical points (outside the affine span) by considering for instance the right singular space of the smallest singular value of $Z(x)$. However, this would depend on the metric on the space of weights and a renormalization of the weights (such as for spheres) can change the smallest non-zero singular value. Positive points are obviously non-degenerate. In Euclidean spaces, all the points of an affine span are positive and non-degenerate.  In positively curved manifolds, we may have degenerate points and non-positive points, as we will see with the sphere example. For negatively curved spaces, the intuition that points of the EBS should all be positive like in Euclidean spaces is also wrong, as we sill see with the example of hyperbolic spaces.

\begin{theorem}[Karcher barycentric subspace and positive span] $   $\\
\label{THM2}
 $EBS^+(x_0, \ldots x_k)$ is the set of  non-degenerate points of the Karcher barycentric subspace $\KBS(x_0, \ldots x_k)$ on $\Mstar$.  In other words, the KBS is the positive EBS plus potentially some degenerate points of the affine span and some points of the cut locus of the reference points. 
\end{theorem}

\subsection{Spherical KBS}
\label{Sec:SphericalKBS}

In order to find the positive points of the EBS on the sphere, we compute the Hessian of the normalized variance. Using Eq.(\ref{eq:HessDistSphere2})  and $u_i= { \log_x(x_i) }/{\theta_i}$, we obtain the Hessian of $\sigma^2(x,\lambda) = \frac{1}{2}\sum_{i=0}^k \nlambda_i \dist^2(x, x_i)$: 
\[\textstyle
H(x, \lambda)  =   \big(\sum_i \nlambda_i  \theta_i  \cot\theta_i \big)(\Id -xx\trp) 
+ \sum_i \nlambda_i  ( 1- \theta_i  \cot\theta_i)u_i u_i\trp.
\]
As expected,  $x$ is an eigenvector with eigenvalue $0$ due to the projection on the tangent space at $x$. Any vector $w$ of the tangent space at $x$ (thus orthogonal to $x$) which is orthogonal to the affine span (and thus to the vectors $u_i$)  is an eigenvector with eigenvalue $ \sum_i \nlambda_i \theta_i \cot \theta_i $. Since the Euclidean affine span $\Aff_{\R^{n+1}}(X)$ has $rank(X) \leq k+1$ dimensions, this eigenvalue has multiplicity $n+1-rank(X) \geq n - k$ when $x\in \Aff(X)$. The last $Rank(X)-1$ eigenvalues have associated eigenvectors within $\Aff_{\R^{n+1}}(X)$.

\cite{buss_spherical_2001} have  have shown that this Hessian matrix is positive definite for {\em positive weights} when the points are within one hemisphere with at least one non-zero weight point which is not on the equator. In contrast, we are interested here in the positivity and definiteness of the Hessian $H(x,\lambda)$ for the positive and negative weights which  live in dual space of barycentric coordinates $\Lambda(x)$. This is actually a non trivial algebraic geometry problem.  Simulation tests with random reference points $X$   show that the eigenvalues of $H(x, \nlambda(x))$  can be positive or negative at different  points of the EBS. The number of positive eigenvalues (the index) of the Hessian is illustrated on Fig. (\ref{fig:signature}) for a few configuration of 3 affinely independent reference points on the 2-sphere. This illustrates the subdivision of the EBS on spheres in a cell complex based on the index of the critical point: the positive points of the KBS do not in general cover the full subsphere containing the reference points. It may even be disconnected, contrarily to the affine span which consistently covers the whole subsphere.  For subspace definition purposes, this suggests that the affine span might thus be the most interesting definition.  For affinely dependent points, the KBS/FBS behave similarly to the EBS. For instance, the weighted variance of  $X=[e_1,-e_1]$ on a 2-sphere is a function of the latitude only. The points of a parallel at any specific latitude  are global minima of the weighted variance for a choice of $\lambda =(\alpha : 1-\alpha), \: \alpha \in [0,1]$. Thus,  all points of the sphere belong to the KBS, which is also the FBS and the affine span. However, the Hessian matrix has one positive eigenvalue along meridians and one zero eigenvalue along the parallels. This is a very non-generic case.  

\begin{figure}
\includegraphics[height=3.1cm]{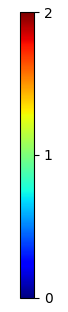}
\includegraphics[height=3.1cm]{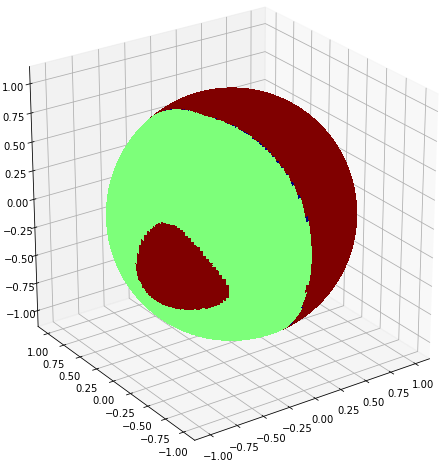}
\includegraphics[height=3.1cm]{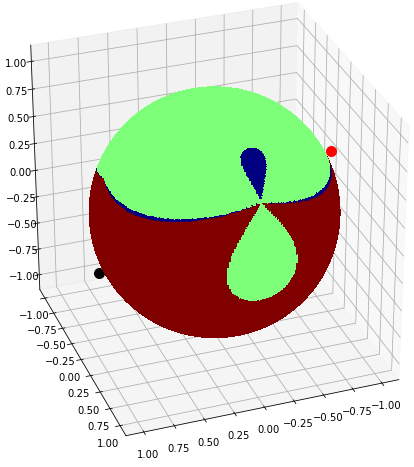}
\includegraphics[height=3.1cm]{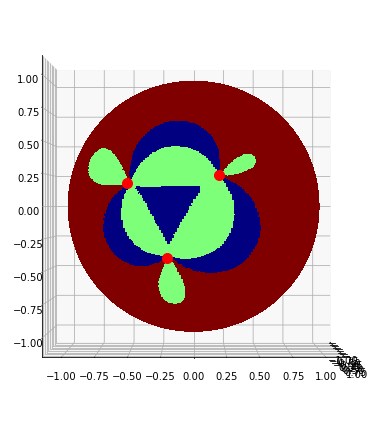}
\includegraphics[height=3.26cm]{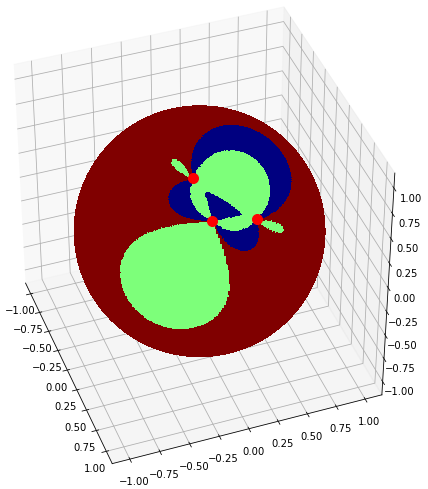}
\caption{Signature of the weighted Hessian matrix for different configurations of 3 reference points (in black, antipodal point in red) on the 2-sphere: the locus of local minima (KBS) in brown does not cover the whole sphere and can even be disconnected (first example).}
\label{fig:signature}
\end{figure}

\subsection{Hyperbolic KBS / FBS}
\label{Sec:HyperbolicKBS}

Let $x =  X \tilde \lambda$ be a point of the hyperbolic affine span of $X=[x_0,\ldots x_k]$. The renormalized weights $\tilde \lambda$ are related to the original weights through $\lambda = F_*(X,x)^{-1} \tilde \lambda$ and satisfy $\|X \tilde{\lambda}\|_*^2 = -1$ and $\text{sgn}( [X \tilde{\lambda}]_0) >0$. The point $x$ is a critical point of the (normalized) weighted variance. In order to know if this is a local minimum (i.e. a point of the KBS), we compute the Hessian of this weighted variance. Denoting $ u_i = \log_x(x_i) / \theta_i$ with  $\cosh \theta_i = -\scal{x}{x_i}_*$, and using the Hessian of the square distance derived in Eq.\eqref{eq:MomemtHyperboloid}, we obtain the following formula:
\[\textstyle
 H(x,\lambda) =  \sum_i \nlambda_i \theta_i \coth \theta_i (J + J x x\trp J) +
  \sum_i \nlambda_i {(1 - \theta_i \coth \theta_i)} J u_i u_i\trp J. 
\]

As expected,  $x$ is an eigenvector with eigenvalue 0 due to the projection on the tangent space at $x$. Any vector $w$ of the tangent space at $x$ which is orthogonal to the affine span (and thus to the vectors $u_i$)  is an eigenvector with eigenvalue $\sum_i \nlambda_i \theta_i \coth \theta_i = 1/(\one\trp \tilde \lambda)$ with multiplicity $n+1-rank(X)$. The last $Rank(X)-1$ eigenvalues have associated eigenvectors within $\Aff_{\R^{n+1}}(X)$. Simulation tests with random reference points $X$ show these eigenvalues  can be positive or negative at different  points of $Aff(X)$. The index of the Hessian is illustrated on Fig. (\ref{fig:signatureHyp}) for a few configuration of 3 affinely independent reference points on the 2-hyperbolic space. Contrarily to the sphere, we observe only one or two positive eigenvalues corresponding respectively to saddle points and local minima. This subdivision of the hyperbolic affine span in a cell complex shows that the hyperbolic KBS is in general a strict subset of the hyperbolic affine span.  We conjecture that there is an exception for reference points at infinity, for which the barycentric subspaces could be generalized using Busemann functions \citep{busemann_geometry_1955}: it is likely that the FBS, KBS and the affine span are all equal in this case  and cover the whole lower dimensional hyperbola.

\begin{figure}
\includegraphics[height=3cm]{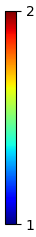}
\includegraphics[height=3cm]{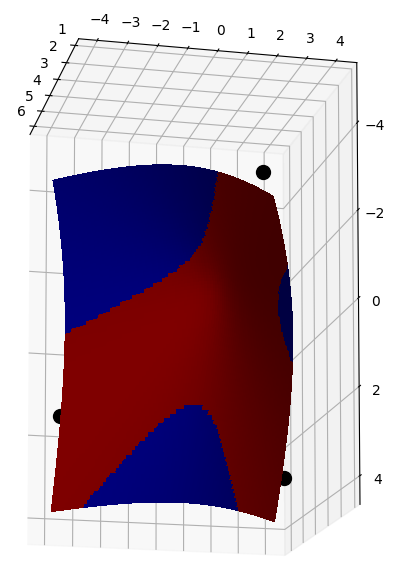}
\includegraphics[height=3cm]{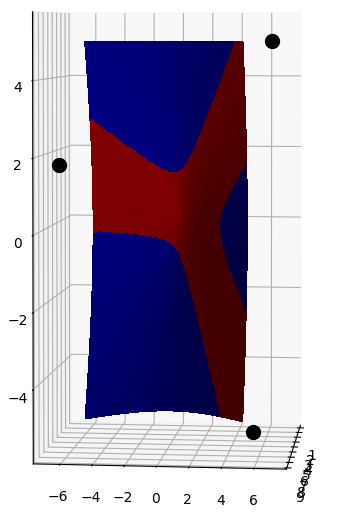}
\includegraphics[height=3cm]{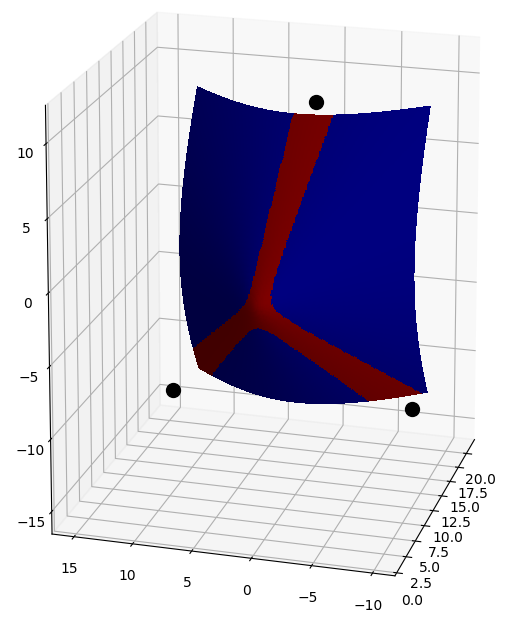}
\includegraphics[height=3cm]{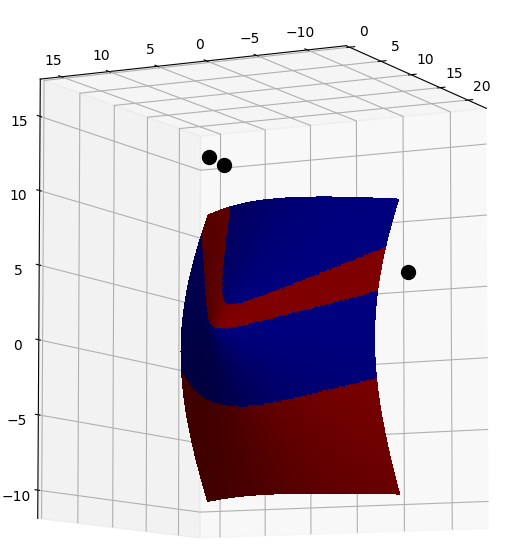}
\includegraphics[height=3cm]{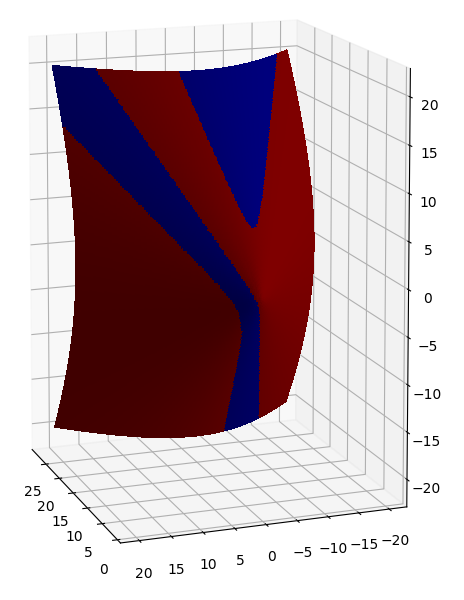}
\caption{Signature of the weighted Hessian matrix for different configurations of 3 reference points on the 2-hyperboloid: the locus of local minima (KBS) in brown does not cover the whole hyperboloid and can be disconnected (last two example).}
\label{fig:signatureHyp}
\end{figure}

\section{Properties of the barycentric subspaces} 
\label{Sec:Prop}

The EBS exists at each reference point $x_i$ with weight 1 for this point and zero for the others. Moreover, when the points are affinely independent, the matrix $Z(x_i)$ has exactly one zero singular value since column $i$ is $\log_{x_i}(x_i)=0$ and all the other column vectors are affinely independent. Finally, the weighted Hessian matrix boils down to $H(x_i,\lambda) =  - \left. D_x \log_{x}(x_i)\right|_{x=x_i} = \Id$ (See e.g. Eq.(\ref{eq:Diff_log})). Thus the reference points are actually local minima of the weighted variance and  the KBS exists by continuity in their neighborhood. 

\subsection{Barycentric simplex in a regular geodesic ball}
We call   the subset of the FBS that has non-negative weights a barycentric simplex. It contains all the reference points, the geodesics segments between the reference points, and of course the Fr\'echet mean of the reference points. This is the generalization of a geodesic segment for 2 points, a triangle for 3 points, etc. The $(k-l)$-faces of a $k$-simplex are the simplices defined by the barycentric subspace of $k-l+1$ points among the $k+1$. They are obtained by imposing the $l$ remaining barycentric coordinates to be zero. In parallel to this paper,  \cite{weyenberg_statistics_2015} has investigated  barycentric simplexes as extensions of principal subspaces in the negatively curved metric spaces of trees under the name Locus of Fr\'echet mean (LFM), with very interesting results.

\begin{theorem}[Barycentric simplex in a regular geodesic ball] \label{THM3}
Let $\kappa$ be an upper bound of sectional curvatures of $\M$ and $\text{inj}(\M)$ be the radius of injection (which can be infinite) of the Riemannian manifold. Let  $X= \{ x_0,\ldots x_k\} \in  \M^{(k+1)}$ be a set of $k+1\leq n$ affinely independent points included in a regular geodesic ball $B(x,\rho)$ with $\rho < \frac{1}{2}\min\{ \text{inj}(\M), \frac{1}{2}\pi/\sqrt{\kappa} \} $ ($\pi/\sqrt{\kappa}$ being infinite if $\kappa < 0$). The barycentric simplex is the graph of a $k$-dimensional differentiable function from the non-negative quadrant of homogeneous coordinates $(\Pk)^+$ to $B(x,\rho)$ and is thus at most $k$-dimensional. The $(k-l)$-faces of the simplex are the simplices defined by the barycentric subspace of $k-l+1$ points among the $k+1$ and include the reference points themselves as vertices and the geodesics joining them as edges.
\end{theorem}
\begin{proof}
The proof closely follows the one of \cite{karcher77} for the uniqueness of the Riemannian barycenter. The main argument is that 
$\mu_{(X, \lambda)}(x) = \sum \nlambda_i \delta_{x_i}(x)$ is a probability distribution whose support is included in the strongly convex geodesic ball $B(x,\rho)$. The variance $\sigma^2(x, \lambda) = \frac{1}{2}\sum_i \nlambda_i d^2(x, x_i)$ is strictly convex on that ball and has  a unique minimum $x_{\lambda} \in B(x,\rho)$, necessarily the weighted Fr\'echet mean.  This proof of the uniqueness of the weighted Fr\'echet mean with non-negative weights was actually already present in \cite{buser_gromovs_1981}. We supplement the proof here by noting that since the Hessian  $H(x_{\lambda}, \lambda) = \sum_i \nlambda_i H_i(x_{\lambda})$ is the convex combination of positive matrices, it is positive definite for all $\lambda \in (\Pk)^+$ in the positive quadrant. Thus the function $x_{\lambda}$ is differentiable thanks to the implicit function theorem: $ D_{\lambda} x_{\lambda}  = H( x_{\lambda}, \lambda)\inv Z(x_{\lambda}).$ The rank of this derivative is at most $k$ since $Z(x_{\lambda})=0$, which proves that the graph of the function $x_{\lambda}$ describes at most a $k$ dimensional subset in $\M$. 
\end{proof}

\subsection{Barycentric simplexes and convex hulls}
In a vector space, a point lies in the convex hull of a simplex if and only if its barycentric coordinates are all non-negative (thus between 0 and 1 with the unit sum constraint). Consequently, barycentric coordinates are often thought to be related to convex hulls. However, in a general Riemannian manifold, the situation is quite different. When there are closed geodesics, the convex hull can reveal several disconnected components, unless one restrict to convex subsets of the manifolds as shown by \cite{Groisser:2003}. In metric spaces with negative curvature (CAT spaces), \cite{weyenberg_statistics_2015} displays explicit examples of convex hulls of 3 points which are 3-dimensional rather than 2-dimensional as expected. In fact, the relationship between barycentric simplexes and convex hulls cannot hold in general Riemannian manifolds if the barycentric simplex is not totally geodesic at each point, which happens for constant curvature spaces but not for general Riemannian manifolds.

\subsection{Local dimension of the barycentric subspaces}

 Let $x$ be a point of the $EBS$ with affinely independent reference points. The EBS equation $Z(x)\lambda = 0$ for $\lambda \in \Lambda(x)$ is smooth in $x$ and $\lambda$ so that we can take a Taylor expansion: at the first order, a variation of barycentric coordinates $\delta \lambda$  induces a variation of position $\delta x$ which are linked through  $H(x,\lambda) \delta x - Z(x) \delta \lambda =0.$ Thus, at regular points:
\[
\delta x = H(x,\lambda)\inv Z(x) \delta \lambda.
\]

Let $Z(x)=U(x)S(x)V(x)\trp$ be a singular value decomposition with singular values sorted in decreasing order. Since $x$ belongs to the EBS, there is at least one (say $m \geq 1$) singular value that vanish and the dual space of barycentric coordinates is  $\Lambda(x) = \Span(v_{k-m}, \ldots v_k)$. For a variation of weights $\delta \lambda$ in this subspace, there is no change of coordinates, while any variation of weights in $\Span(v_0, \ldots v_{k-m-1})$ induces a non-zero position variation.  Thus,  the tangent space of the EBS restricts to the $(k-m)$-dimensional linear space generated by $\{ \delta x_i' =  H(x,\lambda)\inv u_i\}_{0\leq i\leq k-m}$.  Here, we see that the Hessian matrix $H(x, \lambda)$ encodes the distortion of the orthonormal frame fields $ u_1(x), \ldots u_k(x)$ to match the tangent space. Since the lower dimensional subspaces are included one the larger ones, we have a stratification of our $k$-dimensional submanifold into $k-1$, $k-2, \ldots 0$-dimensional subsets. 

\begin{theorem}[Dimension of the exponential barycentric subspace at non-degenerate points] 
\label{THM4}
The non-degenerate exponential barycentric subspace $EBS^*(x_0,\ldots,x_k)$ of $k+1$ affinely independent points is a stratified 
space of dimension $k$ on  $\Mstar$. On the $m$-dimensional strata, $Z(x)$ has exactly $k-m+1$ vanishing singular values.
\end{theorem}

At degenerate points, $H(x, \lambda)$ is not invertible and vectors living in its kernel are also authorized, which potentially raises the dimensionality of the tangent space, even if they do not change the barycentric coordinates. These pathologies do not appear in practice for the constant curvature spaces as we have seen with spherical and hyperbolic spaces, and we conjecture that this is also not the case for symmetric spaces.

\subsection{Stability of the affine span with respect to the metric power}

The Fr\'echet (resp. Karcher) mean can be further generalized by taking a power $p$ of the metric  to define the $p$-variance $\sigma^{p}(x) = \frac{1}{p} \sum_{i=0}^k \dist^{p}(x, x_i)$.  The global (resp. local) minima of this $p$-variance defines the median for $p =1$. This suggest to further generalize barycentric subspaces by taking the locus of the  minima of the weighted $p$-variance $\sigma^{p}(x,\lambda) = \frac{1}{p} \sum_{i=0}^k \nlambda_i \dist^{p}(x, x_i)$. In fact, it turns out that all these "$p$-subspaces"  are necessarily included in the affine span, which shows this notion is really central. To see that, we compute the gradient of the $p$-variance at non-reference point of $\Mstar$:
\[\textstyle
\nabla_x \sigma^{p}(x,\lambda)  = - \sum_{i=0}^k \nlambda_i  \dist^{p -2}(x, x_i) \log_{x}(x_i).
\]
Critical points of the $p$-variance satisfy the equation  $\sum_{i=0}^k \lambda'_i  \log_{x}(x_i) =0$ for the new weights $\lambda'_i = \lambda_i \dist^{p -2}(x, x_i) $. Thus, they are still elements of the EBS and changing the power of the metric just amounts to a reparametrization of the barycentric weights.

\subsection{Restricted geodesic submanifolds are limit of affine spans}

We investigate in this section what is happening when all the points $\{x_i = \exp_{x_0}(\varepsilon  w_i)\}_{1\leq i\leq k}$ are converging to $x_0$ at first order along $k$ independent vectors $\{ w_i\}_{1\leq i\leq k}$.  Here, we fix $w_0 =0$ to simplify the derivations, but the proof can be easily extended with a suitable change of coordinates provided that $\sum_{i=0}^k w_i =0$.  In Euclidean spaces, a point of the affine span $y = \sum_{i=0}^k \nlambda_i x_i$ may be written as the point $y = x + \varepsilon  \sum_{i=1}^k  \nlambda_i w_i$ of the "geodesic subspace" generated by the family of vectors $\{ w_i\}_{1\leq i\leq k}$. By analogy,  we expect the exponential barycentric subspace $\EBS(x_0, \exp_{x_0}(\varepsilon  w_1) \ldots \exp_{x_0}(\varepsilon  w_k))$  to converge towards the totally geodesic subspace at $x$ generated by the $k$ independent vectors $w_1, \ldots w_k$ of $T_x\M$:
\[ \textstyle
GS(x, w_1, \ldots w_k) = \left\{ \textstyle \exp_{x}\left( \sum_{i=1}^k \alpha_i  w_i  \right) \in \M \text{ for } \alpha \in \R^k \right\}.
\] 

In fact, the above definition of the geodesic subspaces (which is the one implicitly used in most of the works using PGA) is too large and may not define a $k$-dimensional submanifold when there is a cut-locus. For instance, it is well known that geodesics of a flat torus are either periodic or everywhere dense in a flat torus submanifold depending on whether the components of the initial velocity field have rational or irrational ratios. This means that the geodesic space generated by a single vector for which all ratio of coordinates are irrational (e.g. $w=(\pi, \pi^2,\ldots \pi^k)$) is filling the full $k$-dimensional flat torus. Thus all the 1-dimensional geodesic subspaces that have irrational ratio of all coordinates minimize the distance to any set of data points in a flat torus of any dimension. In order to have a more meaningful definition and to guaranty the dimensionality of the geodesic subspace, we need to restrict the definition to the points of the geodesics that are distance minimizing.

\begin{definition}[Restricted Geodesic Submanifolds] \label{def:RGS}
Let $x \in \M$ be a point of a Riemannian manifold and let $W_x = \{ \sum_{i=1}^k \alpha_i  w_i,  \alpha \in \R^k\}$ be the $k$-dimensional linear subspace of $T_x\M$ generated  a $k$-tuple $\{ w_i\}_{1\leq i\leq k} \in (T_x\M)^k$ of independent tangent vectors  at $x$. We consider the geodesics starting at $x$ with tangent vectors in $W_x$, but up to the first cut-point of $x$ only. This generates a submanifold of $\M$ called the restricted geodesic submanifold $GS^*(W_x)$:
\[ \textstyle
GS^*(W_x)  = GS^*(x, w_1, \ldots w_k) = \{ \exp_{x}\left( w \right), w\in W_x \cap D(x) \},
\]
where $D(x) \subset T_x\M$ is the injectivity domain. 
\end{definition}
It may not be immediately clear that the subspace we define that way is a submanifold of $\M$: since $\exp_x$ is a diffeomorphism from $D(x) \subset T_x\M$ to $\M \setminus \C(x)$ whose differential has full rank, its restriction to the open star-shape subset $ W_x \cap D(x)$ of dimension $k$ is a diffeomorphism from that subset to the restricted geodesic subspace $GS^*(W_x)$ which is thus an open submanifolds of dimension $k$ of $\M$. This submanifold is generally not geodesically complete.

\begin{theorem}[Restricted geodesic subspaces are limit of affine spans] 
\label{THM6}
The restricted geodesic submanifold $GS^*(W_{x_0}) = \{ \exp_{x_0}\left( w \right), w\in W_{x_0} \cap D(x_0) \}$ is the limit of the $EBS(x_0, x_1(\varepsilon), \ldots x_k(\varepsilon))$  when the points $x_i(\varepsilon) = \exp_{x_0}(\varepsilon  w_i)$ are converging to $x_0$ at first order in $\varepsilon$ along the tangent vectors $w_i$ defining the $k$-dimensional subspace $W_{x_0} \subset T_{x_0}\M$. These limit points are parametrized by barycentric coordinates at infinity in the codimension 1 subspace $\mathds{1}^{\perp}$, the projective completion of $\Pk$ in $\R P^k$, see Definition \ref{def:Pk}.
\end{theorem}

The proof is deferred to Appendix A because of its technicality. We conjecture that the construction can  be generalized using techniques from sub-Riemannian geometry to higher order derivatives when the first order derivative do not span a $k$-dimensional subspace. This would mean that we could also see some non-geodesic decomposition schemes as limit cases of barycentric subspaces, such as splines on manifolds \cite{crouch_dynamic_1995,machado_higher-order_2010,Gay-Balmaz:2012:10.1007/s00220-011-1313-y}.

\paragraph{Example on spheres and hyperbolic spaces}

In spheres (resp. hyperbolic spaces), the restricted geodesic subspace $GS^*(W_{x})$ describes a great subsphere  (resp. a great hyperbola), except for the cut-locus of the base-point $x$ in spheres. Thus, points of $GS^*(W_{x})$ are also points of the affine span generated by $k+1$ affinely independent reference points of this subspace. When all the reference points $x_i = \exp_{x}(\varepsilon w_i)$ coalesce to a single point $x$ along the tangent vectors $W = [w_0,\ldots w_k]$ (with $W \one =0$), we find that solutions of the EBS equation are of the form  $y =  x + W ( \varepsilon \tilde \lambda / \one\trp \tilde \lambda) + O(\varepsilon^2)$, which describes the affine hyperplane generated by $x$ and $W$ in the embedding Euclidean (resp. Minkowski) space. The weights $\mu = \varepsilon \tilde \lambda / \one\trp \tilde \lambda$ converge to points at infinity ($\one\trp \mu =0$) of the affine k-plane of normalized weights.

When reference points coalesce with an additional second order acceleration orthogonally to the subspace $W_x$, we conjecture that the affine span is not any more a great subspheres but a smaller one. This would include principal nested spheres (PNS) developed by \cite{jung_generalized_2010,jung_analysis_2012} as a limit case of barycentric subspaces.  It would be interesting to derive a similar procedure for hyperbolic spaces and to determine which types of subspaces could be obtained by such limits for more general non-local and  higher order jets.

\section{Barycentric subspace analysis}
\label{Sec:BSA}

PCA can be viewed as the search for a sequence of nested linear spaces that best approximate the data at each level. In a Euclidean space, minimizing the variance of the residuals boils down to an independent optimization of orthogonal subspaces at each level of approximation, thanks to the Pythagorean theorem. This enables building each subspace of the sequence by adding (resp. subtracting) the optimal one-dimensional subspace iteratively in a forward (resp. backward) analysis. Of course, this property does not scale up to manifolds, for which  the orthogonality of subspaces is not even well defined.

\subsection{Flags of barycentric subspaces in manifolds}
\cite{damon_backwards_2013} have argued that the nestedness of approximation spaces is one of the most important characteristics for generalizing PCA to more general spaces. Barycentric subspaces can easily be nested, for instance by adding or removing one or several points at a time, to obtains a family of embedded submanifolds which generalizes flags of vector spaces. 

A flag of a vector space $V$  is a filtration of subspaces (an increasing sequence of subspaces, where each subspace is a proper subspace of the next): $\{0\} = V_0 \subset V_1 \subset V_2 \subset \cdots \subset V_k = V$. Denoting  $d_i = \dim(V_i)$ the dimension of the subspaces, we have $0 = d_0 < d_1 < d_2 < \cdots < d_k = n$, where n is the dimension of V. Hence, we must have $k \leq n$. A flag is {\em complete} if $d_i = i$, otherwise it is a {\em partial flag}. Notice that a linear subspace $W$ of $V$ is identified to the partial flag $ \{0\}  \subset W \subset V$. A flag can be generated by adding the successive eigenspaces of an SPD matrix with increasing eigenvalues. If all the eigenvalues have multiplicity one, the generated flag is complete and one can parametrize it by the ordered set of eigenvectors. If an eigenvalue has a larger multiplicity, then the corresponding eigenvectors might be considered as exchangeable in this parametrization in the sense that we should only consider the subspace generated by all the eigenvectors of that eigenvalue.

In an $n$-dimensional manifold $\M$, a strict ordering of $n+1$ independent points $x_0\prec x_1 \ldots  \prec x_n$ defines a filtration of barycentric subspaces. For instance: $\EBS(x_0) = \{ x_0 \}  \subset  \cdots  \EBS(x_0, x_1, x_k) \cdots \subset \EBS(x_0, \ldots x_n).$ The 0-dimensional subspace is now a points in $\M$ instead of the null vector in flags of vector spaces because we are in an affine setting. Grouping points together in the addition/removal process generates a partial flag of barycentric subspaces. Among the barycentric subspaces, the affine span  seems to be the most interesting definition. Indeed, when the manifold $\Mstar$ is connected, the EBS of $n+1$ affinely independent  points covers the full manifold $\Mstar$, and its completion covers the original manifold: ${\Aff}(x_0,\ldots x_n) = \M$. With the Fr\'echet or Karcher barycentric subspaces,  we only generate a submanifold (the positive span) that does not cover the whole manifold in general, even in negatively curved spaces.

\begin{definition}[Flags of affine spans in manifolds] 
Let  $x_0\preceq x_1 \ldots  \preceq x_k$ be $k+1 \leq n+1$ affinely independent ordered  points of $\M$ where two or more successive points are either strictly ordered ($x_i \prec x_{i+1}$) or exchangeable ($x_i \sim x_{i+1}$). For a strictly ordered set of points, we call  the sequence of properly nested subspaces $FL_i(x_0\prec x_1 \ldots  \prec x_k) = {\Aff}(x_0, \ldots x_i)$ for $0 \leq i \leq k$ the flag of affine spans $FL(x_0\prec x_1 \ldots  \prec x_k)$. For a flag comprising exchangeable points, the different subspaces of the sequence are only generated at strict ordering signs or at the end. A flag is said complete if it is strictly ordered with $k=n$. We call  a flag  of exchangeable points $FL(x_0\sim x_1 \ldots  \sim x_k)$ a pure subspace because the sequence is reduced to the unique subspace $FL_k(x_0\sim x_1 \ldots \sim x_k) = {\Aff}(x_0, \ldots x_k)$.
\end{definition}

\subsection{Forward and backward barycentric subspaces analysis}

In Euclidean PCA, the flag of linear subspaces can be built in a forward way, by computing the best 0-th order approximation (the mean), then the best first order approximation (the first mode), etc. It can also be built backward, by removing the direction with the minimal residual  from the current affine subspace. In a manifold, we can use similar forward and backward analysis, but they have no reason to give the same result.

With a forward analysis, we compute iteratively the flag of affine spans by adding one point at a time keeping the previous ones fixed. The barycentric subspace $\Aff(x_0) = \{ x_0 \}$ minimizing the unexplained variance is a Karcher mean. Adding a second point amounts to compute the geodesic passing through the mean that best approximate the data. Adding a third point now differ from PGA, unless the three points coalesce to a single one. With this procedure, the Fr\'echet mean always belong to the barycentric subspace. 

The backward analysis consists in iteratively removing one dimension. One should theoretically  start with a full set of points  and chose which one to remove. However, as  all the sets of $n+1$ affinely independent points generate the full manifold with the affine span, the optimization really begin with the set of $n$ points $x_0, \ldots x_{n-1}$. We should afterward only test for which of the $n$ points we should remove. Since optimization is particularly inefficient in large dimensional spaces, we may run a forward analysis until we reach the noise level of the data for a dimension $k \ll n$. In practice, the noise level is often unknown and a threshold at 5\% of the data variance is sometimes chosen. More elaborate methods exist to determine the intrinsic dimension of the data for manifold learning technique \citep{wang_scale-based_2008}. Point positions may be optimized at each step to find the optimal subspace and a backward sweep reorders the points at the end. With this process, there is no reason for the Fr\'echet mean to belong to any of the barycentric subspaces. For instance, if we have clusters, one expects the reference points to localize within these clusters rather than at the Fr\'echet mean.

\subsection{Approximating data using a pure subspace}

Let ${Y} = \{ \hat y_i \}_{i=1}^N \in \M^N$ be $N$ data points  and  $X=\{x_0,\ldots x_k\}$ be $k+1$ affinely independent reference points. We assume that each data point $\hat y_i$ has almost surely  one unique closest point $y_i(X)$ on the barycentric subspace. This is the situation for Euclidean, hyperbolic and spherical spaces, and this should hold more generally for all the points outside the focal set of the barycentric subspace. This allows us to write the residual $r_i(X) = \dist( \hat y_i,y_i(X))$ and to consider the minimization of the unexplained variance $\sigma^2_{out}(X) = \sum_j r_i^2(X)$. This optimization problem on $\M^{k+1}$ can be achieved by standard techniques of optimization on manifolds (see e.g. \cite{OptimizationManifold:2008}). However, it is not obvious that the canonical product Riemannian metric is the right metric to use, especially close to coincident points. In this case, one would like to consider switching to the space of (non-local) jets to guaranty the numerical stability of the solution. In practice, though, we may constraint the distance between reference points to be larger than a threshold.

A second potential problem is the lack of identifiability: the minimum of the unexplained variance may be reached by subspaces parametrized by several k-tuples of points. This is the case for constant curvature spaces since every linearly independent $k$-tuple of points in a given subspace parametrizes the same barycentric subspace. In constant curvature spaces, this can be accounted for using a suitable polar or QR matrix factorization (see e.g. \ref{suppB}). In general manifolds, we expect that the absence of symmetries will break the multiplicity of this relationship (at least locally) thanks to the curvature. However, it can lead to very badly conditioned systems to solve from a numerical point of view for small curvatures.
 
A last problem is that the criterion we use here (the unexplained variance) is only valid for a pure subspace of fixed dimension, and considering a different dimension will lead in general to pure subspaces which cannot be described by a common subset of reference points. Thus, the forward and backward optimization of nested barycentric subspaces cannot lead to the simultaneous optimality of all the subspaces of a flag in general manifolds.

\subsection{A criterion for hierarchies of subspaces: AUV on flags of affine spans}

In order to obtain consistency across dimensions, it is necessary to define a criterion which depends on the whole flag of subspaces and not on each of the subspaces independently. In PCA, one often plots the unexplained variance as a function of the number of modes used to approximate the data. This curve should decreases as fast as possible from the variance of the data (for 0 modes) to 0 (for $n$ modes). A standard way to quantify the decrease consists in summing the values at all steps, giving the Accumulated Unexplained Variances (AUV), which is analogous to the Area-Under-the-Curve (AUC) in Receiver Operating Characteristic (ROC) curves.

Given a strictly ordered flag of affine subspaces $Fl(x_0\prec x_1 \ldots  \prec x_k)$, we thus propose to optimize the AUV criterion:
\[ \textstyle
AUV(Fl(x_0\prec x_1 \ldots  \prec x_k)) = \sum_{i=0}^k \sigma^2_{out}( Fl_i(x_0\prec x_1 \ldots  \prec x_k ) )
\]
instead of the unexplained variance at order $k$. We could of course consider a complete flag but in practice it is often useful to stop at a dimension $k$  much smaller than the possibly very high dimension $n$. The criterion is extended to more general partial flags by weighting the unexplained variance of each subspace  by the number of (exchangeable) points that are added at each step. With this global criterion, the point $x_i$ influences all the subspaces of the flag that are larger than $Fl_i(x_0\prec x_1 \ldots  \prec x_k )$ but not the smaller subspaces. It turns out that optimizing this criterion results in the usual PCA up to mode $k$ in a Euclidean space.

\begin{theorem}[Euclidean PCA as an optimization in the flag space] \label{THM8}
Let ${\hat Y} = \{ \hat y_i \}_{i=1}^N$ be a set of $N$ data points in $\R^n$. We denote as usual the mean by $\bar y = \frac{1}{N} \sum_{i=1}^N  \hat y_i$ and the empirical covariance matrix by $\Sigma = \frac{1}{N} \sum_{i=1}^N  (\hat y_i -\bar y) (\hat y_i -\bar y)\trp$. Its spectral decomposition is denoted by $\Sigma = \sum_{j=1}^n \sigma_j^2 u_j u_j\trp$ with the eigenvalues sorted in decreasing order. We assume that the first $k+1$ eigenvalues have multiplicity one, so that the order from $\sigma_1$ to $\sigma_{k+1}$ is strict.  

Then the partial flag of affine subspaces $Fl(x_0\prec x_1 \ldots  \prec x_k)$  optimizing 
\[ \textstyle
AUV(Fl(x_0\prec x_1 \ldots  \prec x_k)) = \sum_{i=0}^k \sigma^2_{out}( Fl_i(x_0\prec x_1 \ldots  \prec x_k ) )
\]
is strictly ordered  and can be parametrized by $x_0 = \bar y$, $x_i = x_0 + u_i$ for $1 \leq i \leq k$. The parametrization by points is not unique but the flag of subspaces which is generated is and is equal to the flag generated by the PCA  modes up to mode $k$ included.
\end{theorem}
The proof is detailed in \ref{suppB}. The main idea is to parametrize the matrix of reference vectors by the product of an orthogonal  matrix $Q$ with a positive definite triangular superior matrix (QR decomposition). The key property of this Gram-Schmidt orthogonalization is the stability of the columns of $Q$ when we add or remove columns (i.e reference points) in $X$, which allows to write the expression of the AUV explicitly. Critical points are found for columns of $Q$ which are eigenvectors of the data covariance matrix and the expression of the AUV shows that we have to select them in the decreasing order of eigenvalues.

\subsection{Sample-limited barycentric subspace inference on spheres}

In several domains, it has been proposed to limit the inference of the Fr\'echet mean to the data-points only. In neuroimaging studies, for instance, the individual image minimizing the sum of square deformation distance to other subject images has been argued to be a good alternative to the mean template (a Fr\'echet mean in deformation and intensity space) because it conserves the full definition and all the original characteristics of a real subject image \citep{lepore:inria-00616172}. Beyond the Fr\'echet mean, \cite{Feragen2013} proposed to define the first principal component mode as the geodesic going through two of the data points which minimizes the unexplained variance. The method named {\em set statistics} was aiming to accelerate the computation of statistics on tree spaces. \cite{Zhai_2016} further explored this idea under the name of {\em sample-limited geodesics} in the context of PCA in phylogenetic tree space. However, in both cases, extending the method to higher order principal modes was considered as a challenging research topic.

With barycentric subspaces, sample-limited statistics naturally extends to any dimension by restricting the search to (flags of) affine spans that are parametrized by data points. Moreover, the implementation boils down to a very simple enumeration problem. 
An important advantage for interpreting the modes of variation is that reference points are never interpolated as they are by definition sampled from the data. Thus,  we may go back to additional information about the samples like the disease characteristics in medical image image analysis.  The main drawback is the combinatorial explosion of the computational complexity: the optimal order-k flag of affine spans requires $O(N^{k+1})$ operations, where $N$ is the number of data points. In practice, the search can be done exhaustively for a small number of reference points but an approximated optimum has to be sought for larger $k$ using a limited number of random tuples \citep{Feragen2013}.

\begin{figure}[!tb]
\includegraphics[width=0.30\columnwidth]{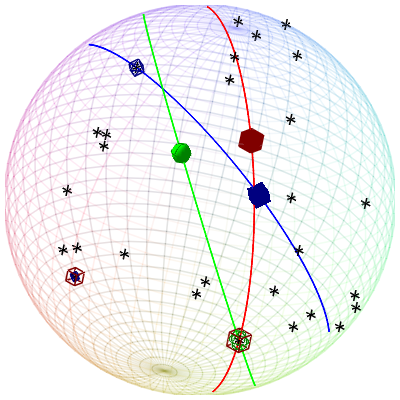}\:
\includegraphics[width=0.37\columnwidth]{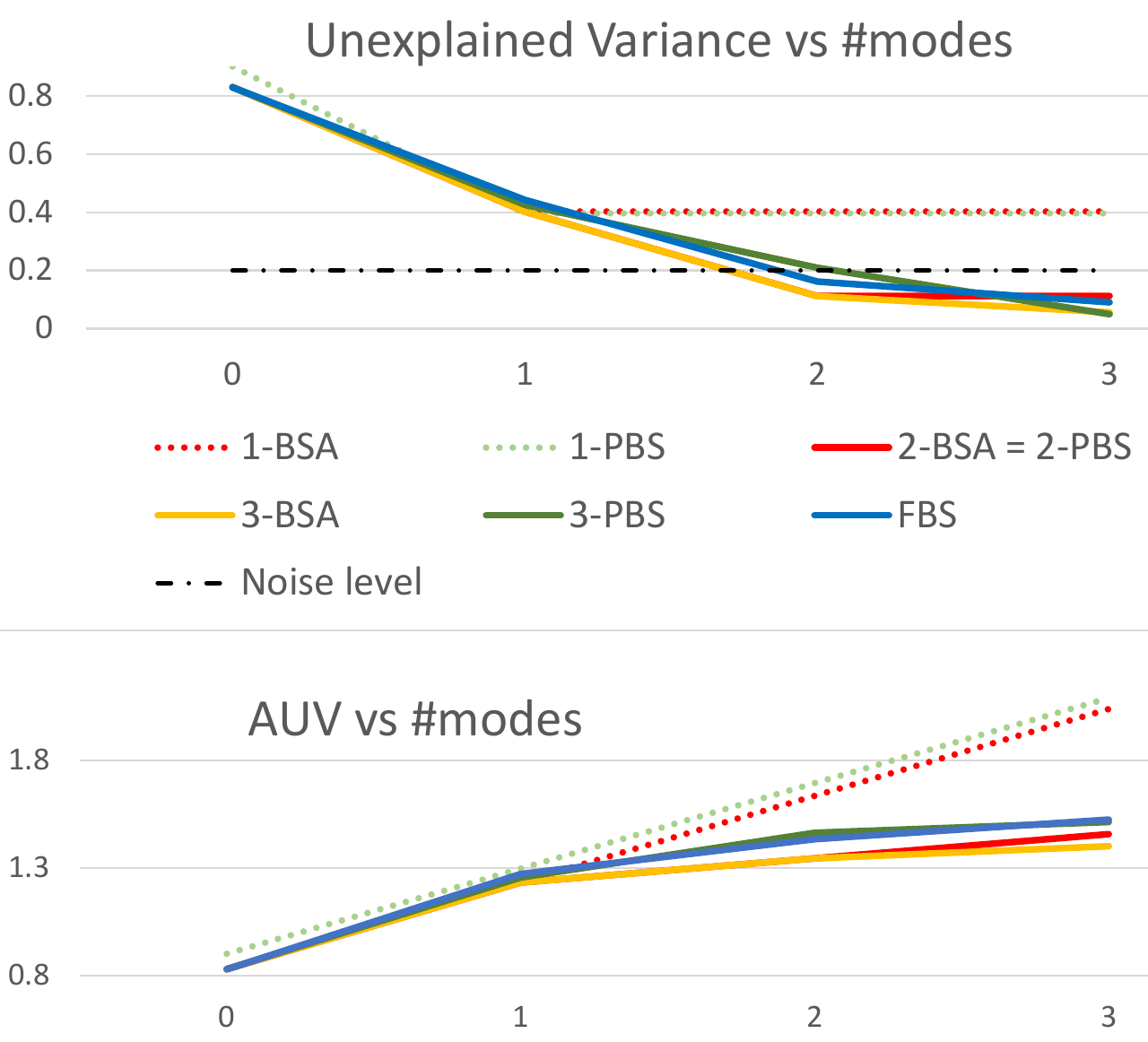}\:
\includegraphics[width=0.30\columnwidth]{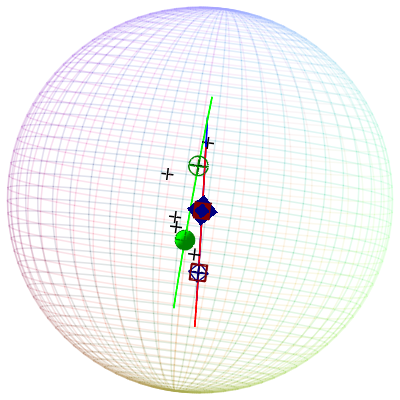}
\caption{{\bf Left:} Equi 30 simulated dataset. Data and reference points are projected from the 5-sphere to the expected 2-sphere in 3d to allow visualization. For each method (FBS in blue, 1-PBS in green and 1-BSA in red), the first reference point has a solid symbol. The 1d mode is the geodesic joining this point to the second reference point. The third reference point of FBS and 2-BSA (on the lower left part) is smaller. {\bf Middle:} graph of the unexplained variance and AUV for the different methods on the Equi 30 dataset. {\bf Right:} Mount Tom Dinosaur trackway 1 data with the same color code. 1-BSA (in red) and FBS (in blue) are superimposed.}
\label{Fig:Equi30}
\end{figure}

In this section, we consider the exhaustive sample-limited version of the Forward Barycentric Subspace (FBS) decomposition,  the optimal $k$-dimensional Pure Barycentric Subspace with backward ordering (k-PBS), and the Barycentric Subspace Analysis up to order k (k-BSA). In order to illustrate the differences, we consider a  first synthetic dataset where we draw 30 random points uniformly on an equilateral triangle of side length $\pi/2$ on a 6-dimensional sphere. We add to each point a (wrapped) Gaussian noise of standard deviation $\sigma = 10^{\circ}$. In this example, original data live on a 2-sphere: the ideal flag of subspaces is a pure 2d subspace spanning the first three coordinates. We illustrate in Fig.\ref{Fig:Equi30} the different reference points that are found for the different methods. We can see that all methods end-up with different results, contrarily to the Euclidean case. The second observation is that the optimal pure subspace is not stable with the dimension: the reference points of the 0-PBS (the sample-limited Fr\'echet mean represented by the large blue solid diamond), the 1-PBS (in green) and the 2-PBS (identical to the red points of the 2-BSA in red)  are all different.  BSA is more stable:  the first reference points are the same from the 1-BSA to the 3-BSA. In terms of unexplained variance, the 2-BSA is the best for two modes (since it is identical to the optimal 2-PBS) and reaches the actual noise level. It remains better than the 3-PBS and the FBS with three modes in terms of AUV even without adding a fourth point. 

As a second example, we take real data encoding the shape of three successive footprints of Mount Tom Dinosaur trackway 1 described in \citep[p.181]{small96}. For planar triangles, the shape space (quotient of the triad by similarities) boils down to the sphere of radius $1/2$. These data are displayed on the right of Fig.\ref{Fig:Equi30}. In this example, the reference points of the 0-BSA to the 3-BSA are stable and identical to the ones of the FBS.  This is a behavior that we have observed in most of our simulations when  modes cannot be confused. This may not hold anymore if reference points were optimized on the sphere rather than on the data points only. The optimal 1-PBS (the best geodesic approximation) picks up different reference points.

\section{Discussion}

We investigated in the paper several notions of subspaces in manifolds generalizing the notion of affine span in a Euclidean space. The Fr\'echet / Karcher / exponential barycentric subspaces are the nested locus of weighted Fr\'echet / Karcher / exponential barycenters with positive or negative weights summing up to 1. The affine spans is the metric completion of the largest one (the EBS). It may be a non-connected manifold with boundaries. The completeness of the affine span enables reconnecting part of the subspace that arrive from different directions at the cut-locus of reference points if needed. It also ensures that there exits a closest point on the submanifold for data projection purposes, which is fundamental for dimension reduction purposes. The fact that modifying the power of the metric does not change the affine span is an unexpected stability result which suggests that the notion is quite central. Moreover, we have shown that the affine span encompass principal geodesic subspaces as limit cases. It would be interesting to show that we can obtain other types of subspaces like principal nested subspheres with higher order and non-local jets: some non-geodesic decomposition schemes such as loxodromes and splines could probably also be seen as  limit cases of barycentric subspaces.

Future work will address barycentric subspaces in interesting non-constant curvatures spaces. For instance, \cite{eltzner_dimension_2015} adaptively deforms the flat torus seen as a product of spheres into a unique sphere to allow principal nested spheres (PNS) analysis. A quick look at the flat torus shows that the the cut-locus of $k+1\leq n$ points in ${\cal S}_1^n$ divides the torus into $k^n$ cells in which the affine span is a $k$-dimensional linear subspace. The subspaces generated in each cell are generally disconnected, but when points coalesce with each others into a jet, the number of cells decreases in the complex and at the limit we recover a single cell that contain a connected affine span. For a first order jet, we recover as expected the restricted geodesic subspace (here a linear subspace limited to the cut locus of the jet base-point), but higher order jets may generate more interesting curved subspaces that may better describe the data geometry. 

The next practical step is obviously the implementation of generic algorithms to optimize barycentric subspaces in general Riemannian manifolds. Example algorithms include: finding a point with given barycentric coordinates (there might be several so this has to be a local search); finding the closest point (and its coordinates) on the barycentric subspace; optimizing the reference points to minimize the residual error after projection of data points, etc. If such algorithms can be designed relatively simply for simple specific manifolds as we have done here for constant curvature spaces, the generalization to general manifolds requires a study of the focal set of the barycentric subspaces or guarantying the correct behavior of algorithms. We conjecture that this is a stratified set of zero measure in generic cases. Another difficulty is linked to the non-identifiability of the subspace parameters. For constant curvature spaces, the right parameter space is actually the $k$-Grassmanian. In more general manifolds, the curvature and the interaction with the cut-locus break the symmetry of the barycentric subspaces, but lead to a poor numerical conditioning of the system good renormalization techniques need to be designed to guaranty the numerical stability. 

Finding the subspace that best explain the data is an optimization problem on manifolds. This raises the question of which metric should be considered on the space of barycentric subspaces. In this paper, we mainly see this space as the configuration space of $k+1$ affinely independent points, with convergence to spaces of jets (including non-local jets) when several points coalesce. Such a construction was named Multispace by \cite{olver_geometric_2001} in the context of symmetry-preserving numerical approximations to differential invariants. It is likely that similar techniques could be investigated to construct numerically stable implementations of barycentric subspaces of higher order parametrized by non-local jets, which are needed to optimize safely. Conversely, barycentric subspaces could help shedding a new light on the multispace construction for  differential invariants. 

Barycentric subspaces could probably be used to extend methods like the probabilistic PCA of \cite{tipping_probabilistic_1999}, generalized to PGA by \cite{zhang_probabilistic_2013}. A first easy step in that direction is to replace the reference points by reference distributions on the manifold and to look at the locus of weighted expected means. Interestingly, this procedure soften the constraints that we had in this paper about the cut locus. Thus, following \cite{karcher77}, reference distributions could be used in a mollifier smoothing approach to study the regularity of the barycentric subspaces.

For applications where data live on Lie groups, generalizing barycentric subspaces to more general non-Riemannian spaces like affine connection manifolds is a particularly appealing extension. In computational anatomy, for instance, deformations of shapes are lifted to a group of diffeomorphism for statistical purposes (see e.g. \cite{lorenzi:hal-00813835,lorenzi:hal-01145728}). All Lie groups can be endowed with a bi-invariant symmetric Cartan-Schouten connection for which geodesics are the left and right translation of one-parameter subgroups. This provides the Lie group with an affine connection structure which may be metric or not. When the group is the direct product of compact and Abelian groups, it admits a bi-invariant metric for which the Cartan-Schouten connection is the natural Levi-Civita connection. Other groups do not admit any bi-invariant metric (this is the case for rigid transformations in more than 2 dimensions because of the semi-direct product), so that a Riemannian structure can only be left or right invariant but not both.
However the bi-invariant Cartan-Schouten connection continues to exists, and one can design bi-invariant means using exponential barycenter as proposed by  \cite{pennec:hal-00699361}. Thus, we may still define exponential barycentric subspaces and affine spans in these affine connection spaces, the main difference being that the derivative of the log is not any more the Hessian of a distance function. This might considerably complexify the analysis of the generated subspaces. 

The second topic of this paper concerns the generalization of PCA to manifolds using Barycentric Subspace Analysis (BSA). \cite{damon_backwards_2013} argued  that an interesting generalization of PCA should rely on “nested sequence of relations”, like embedded linear subspaces in the Euclidean space or embedded spheres in PNS. Barycentric subspaces can naturally be nested by adding or removing points or equivalently by setting the corresponding barycentric coordinate to zero. Thus we can easily generalize PCA to manifolds using a forward analysis by iteratively adding one or more points at a time. At the limit where points coalesce at the first order, this amounts to build a flag  of (restricted) principal geodesic subspaces. Thus it generalizes the Principal Geodesic Analysis (PGA) of \cite{fletcher_principal_2004,sommer_optimization_2013} when starting with a zeroth dimensional space (the Fr\'echet mean) and the Geodesic PCA (GPCA) of \cite{huckemann_principal_2006,huckemann_intrinsic_2010} when starting directly with a first order jet defining a geodesic. One can also design a backward analysis by starting with a large subspace and iteratively removing one or more points to define embedded subspaces. 

However, the greedy optimization of these forward/backward methods generally leads to different solutions which are not optimal for all subspace jointly. The key idea is to consider PCA as a joint optimization of the whole flag of subspaces instead of each subspace independently. In a Euclidean space, we showed that the Accumulated Unexplained Variances (AUV) with respect to all the subspaces of the hierarchy (the area under the curve of unexplained variance) is a proper criterion on the space of Euclidean flags. We proposed to extend this criterion to barycentric subspaces in manifolds, where an ordering of the reference points naturally defines a flag of nested barycentric subspaces. A similar idea could be used with other iterative least-squares methods like partial least-squares (PLS) which are also one-step at a time minimization methods.

\section*{Acknowledgments}
This work was partially supported by the Erwin Schr\"odinger Institute in Vienna through a three-weeks stay in February 2015 during the program Infinite-Dimensional Riemannian Geometry with Applications to Image Matching and Shape. It was also partially supported by the Inria Associated team GeomStats between Asclepios and Holmes' lab at Stanford Statistics Dept. I would particularly like to thank Prof. Susan Holmes for fruitful discussions during the writing of the paper.

\begin{supplement}[id=suppA]
\sname{Supplement A}  
\stitle{Hessian of the Riemannian squared distance}
\slink[doi]{COMPLETED BY THE TYPESETTER}
\sdatatype{.pdf}
\sdescription{This supplementary material describes in more length the notions of Riemannian geometry that are underlying the main paper and investigates the Hessian of the Riemannian square distance whose eigenvalues control the local regularity of the barycentric subspaces. This is exemplified on the sphere and the hyperbolic space.}
\end{supplement}

\begin{supplement}[id=suppB]
\sname{Supplement B} 
\stitle{PCA as an optimization on the flag manifold}
\slink[doi]{COMPLETED BY THE TYPESETTER}
\sdatatype{.pdf}
\sdescription{This supplementary material details in length the proof that the flag of linear subspaces found by PCA optimizes the Accumulated Unexplained Variances (AUV) criterion in a Euclidean space.}
\end{supplement}

\bibliographystyle{imsart-nameyear}


\section*{Appendix A: Proof of Theorem~\ref{THM6}}
\label{ProofTHM6}

\begin{proof}
We first establish a useful formula exploiting the symmetry of the geodesics from $x$ to $y \not \in \C(x)$ with respect to time. Reverting time along a geodesic, we have: $\gamma_{(x,\lcp{xy})}(t) = \gamma_{(y,\lcp{yx})}(1-t)$, which means in particular that $\dot \gamma_{(x,\lcp{xy})}(1) = - \dot \gamma_{(y,\lcp{yx})}(0) = -\lcp{yx}$. Since $\gamma_{(x,\lcp{xy})}(t) = \exp_x(t \lcp{xy})$, we obtain $ \lcp{yx} = - D \left. \exp_x \right|_{\lcp{xy}} \lcp{xy}.$ Now, we also have $ \left( D \left. \exp_x \right|_{\lcp{xy}} \right).  D \left. \log_x \right|_y = \Id$ because $\exp_x( \log_x(y)) = y$. Finally, $D\exp_x$ and $D\log_x$ have full rank on $\M/\C(x)$ since there is no conjugate point before the cut-locus, so that we can multiply by their inverse and we end up with:
\begin{equation}
\label{eq:symgeo}
\forall y \not \in \C(x), \quad \lcp{xy} = - D \left. \log_x \right|_y \lcp{yx}.
\end{equation}

Let us first restrict to a convenient domain of $\M$: we consider a open geodesic  ball $B(x_0, \zeta)$ of radius $\zeta$ centered at $x_0$ and we exclude all the points of $\M$ which cut locus intersect this ball, or equivalently the cut-locus of all the points of this ball. We obtain an open domain ${\cal D}_{\zeta}(x_0) = \M \setminus \C(B(x_0, \zeta))$ in which $\log_x(y)$ is well defined and smooth for all $x \in  B(x_0, \zeta)$ and all $y\in {\cal D}_{\zeta}(x_0)$. Thanks to the symmetry of the cut-locus, $\log_y(x)$ is also well defined and smooth in the same conditions and Eq. (\ref{eq:symgeo}) can be rephrased:
\begin{equation}
\label{eq:symgeo2}
\forall x \in  B(x_0, \zeta), y\in {\cal D}_{\zeta}(x_0), \quad \lcp{xy} = - D \left. \log_x \right|_y \lcp{yx}.
\end{equation}

Let $\|w\|_{\infty} = \max_i \|w_i\|_{x_0}$ be the maximal length of the vectors $w_i$. For $\varepsilon < \zeta / \|w\|_{\infty}$, we have $\|\varepsilon w_i\|_{x_0} \leq \varepsilon \|w\|_{\infty} < \zeta$, so that all the points $x_i = \exp_{x_0}( \varepsilon w_i)$ belong to the open geodesic  ball $B(x_0, \zeta)$. Thus, $\log_x(x_i)$ and $\log_{x_i}(x)$ are well defined and smooth for any $x \in {\cal D}_{\zeta}(x_0)$, and we can write  the Taylor expansion  in a normal coordinate system at $x_0$using Eq.\ref{eq:symgeo2}:
\[\textstyle
\log_x(x_i(\varepsilon))
	= \log_x( x_0) + \varepsilon   D\log_x|_{x_0} w_i + O(\varepsilon ^2)
	=  D\log_x|_{x_0} \left( \varepsilon w_i - \log_{x_0}(x)  \right) + O(\varepsilon ^2)
\]

Any point $x \in {\cal D}_{\zeta}(x_0)$ can be defined by $\log_{x_0}(x) = \sum_{j=1}^k \alpha_i w_i + w_{\bot}$ with $\scal{w_{\bot}}{w_i} =0$ and suitable constraints on the $\alpha_i$ and $w_{\bot}$. Replacing $\log_x( x_0)$ by its value in the above formula, we get
\[\textstyle
\log_x(x_i(\varepsilon))  = D\log_x|_{x_0} \left( \varepsilon w_i - \sum_{j=1}^k \alpha_j w_j - w_{\bot} \right) + O(\varepsilon ^2).
\]
Since the matrix $D\log_x|_{x_0}$ is invertible, the EBS equation $\mathfrak{M}_1(x, \lambda)= \sum_{i=0}^k \lambda_i  \lcp{x x_i}  =0$ is equivalent to $\textstyle \textstyle w_{\bot} + \sum_{j=1}^k \alpha_j w_j - \varepsilon\left(\sum_{i=1}^k \nlambda_i w_i\right) = O(\varepsilon ^2).$ Projecting orthogonally to $W_{x_0}$, we get $w_{\bot} = O(\varepsilon^2)$: this means that any point of the limit EBS has to be of the form $x = \exp_{x_0}(\sum_{j=1}^k \alpha_i w_i)$. In other words, only points of the restricted geodesic subspace  $GS^*(W_{x_0})$ can be solutions of the limit EBS equation.

Now, for a point of $GS^*(W_{x_0})$ to be a solution of the limit EBS equation, there should exists barycentric coordinates $\lambda$ such that $\sum_{j=1}^k (\alpha_j - \varepsilon \nlambda_i) w_j = O(\varepsilon ^2)$. Choosing $\lambda = (\varepsilon - \sum_i \alpha_i: \alpha_1 : \ldots : \alpha_k)$, we obtain the normalized barycentric coordinates $\nlambda_i = \alpha_i / \varepsilon$ for $1\leq i \leq k$ and $\nlambda_0 = 1 - (\sum_i \alpha_i) / \varepsilon$ that satisfy this condition. Thus any point of $GS^*(W_{x_0}) \cap {\cal D}_{\zeta}(x_0)$ is a solution of the limit EBS equation with barycentric coordinates at infinity on $\Pk$. Taking $\zeta$ sufficiently small, we can include all the points of $GS^*(W_{x_0})$.

\end{proof}


\clearpage

\setcounter{equation}{0}
\setcounter{figure}{0}
\setcounter{table}{0}
\setcounter{page}{1}
\setcounter{section}{0}
\setcounter{subsection}{0}%
\setcounter{theorem}{0}
\setcounter{definition}{0}
\setcounter{proposition}{0}
\makeatletter
\renewcommand{\thetheorem}{A\arabic{theorem}}
\renewcommand{\thepage}{A\arabic{page}}
\renewcommand{\thesection}{A\arabic{section}}
\renewcommand{\theequation}{A\arabic{equation}}
\renewcommand{\thefigure}{\arabic{figure}}
\renewcommand{\bibnumfmt}[1]{[A#1]}
\renewcommand{\citenumfont}[1]{A#1}

\begin{frontmatter}

\pdfbookmark[0]{Supplementary Materials A: Hessian of the Riemannian Squared Distance}{SupA}

\title{Supplementary Materials A:\\Hessian of the Riemannian Squared Distance}

\runtitle{Hessian of the Riemannian Squared Distance}

\begin{aug}


\author{\fnms{Xavier} \snm{Pennec}\ead[label=e1A]{xavier.pennec@inria.fr}}
\address{Asclepios team, Inria Sophia Antipolis \\ 2004 Route des Lucioles, BP93
\\ F-06902 Sophia-Antipolis Cedex, France \printead{e1A}}
\affiliation{Universit\'e C\^ote d'Azur and Inria Sophia-Antipolis M\'editerran\'ee}

\runauthor{X. Pennec}
\end{aug}


\begin{abstract}
This supplementary material details the notions of Riemannian geometry that are underlying the  paper {\em Barycentric Subspace Analysis on Manifolds}. In particular, it investigates the Hessian of the Riemannian square distance whose definiteness controls the local regularity of the barycentric subspaces. This is exemplified on the sphere and the hyperbolic space.

\end{abstract}

\end{frontmatter}

\section{Riemannian manifolds}

A Riemannian manifold is a differential manifold endowed with a smooth collection of scalar products $\scal{.}{.}_{x}$  on each tangent space $T_{x}\M$ at point $x$ of the manifold, called the Riemannian metric. In a chart, the metric is expressed by a symmetric positive definite matrix $G(x) = [ g_{ij}(x) ]$ where each element is given by the dot product of the tangent vector to the coordinate curves: $g_{ij}(x) = \scal{\partial_i}{\partial_j}_x$. This matrix is called the {\em local representation of the Riemannian metric} in the chart $x$ and the dot products of two vectors $v$ and $w$ in $T_{x}\M$ is now $\scal{v}{w}_x = v\trp \: G(x)\: w = g_{ij}(x) v^i w^j$ using the Einstein summation convention which implicitly sum over the indices that appear both in upper position (components of [contravariant] vectors) and lower position (components of covariant vectors (co-vectors)).

\subsection{Riemannian distance and geodesics}
If we consider a curve  $\gamma(t)$ on the manifold, we can compute at each point its instantaneous speed vector $\dot{\gamma}(t)$ (this operation only involves the differential structure) and its norm $ \left\| \dot{\gamma}(t)\right\|_{\gamma(t)}$ to obtain the instantaneous speed (the Riemannian metric is needed for this operation).  To compute the length of the curve, this value is integrated along the curve:
\[
  \label{curve_length}
  {\cal L}_a^b (\gamma) = \int_a^b \left\| \dot{\gamma}(t)\right\|_{\gamma(t)}
  dt = 
\int_a^b \left( \scal{ \dot{\gamma}(t)}{\dot{\gamma}(t)
}_{\gamma(t)} \right)^{\frac{1}{2}}dt 
\]
The distance between two points of a connected Riemannian manifold is the minimum length among the curves joining these points.  The curves realizing this minimum are called geodesics. Finding the curves realizing the minimum length is a difficult problem as any time-reparameterization is authorized. Thus one rather defines the metric geodesics as the critical points of the energy functional ${\cal E}(\gamma) = \frac{1}{2}\int_0^1 \left\| \dot \gamma (t)\right\|^2\: dt$. It turns out that they also optimize the length functional but they are moreover parameterized proportionally to arc-length.

Let $[g^{ij}] = [g_{ij}]\inv$ be the inverse of the metric matrix (in a given coordinate system) and $\Gamma^i_{jk} = \frac{1}{2} g^{im}\left( \partial_k g_{mj} + \partial_j g_{mk} - \partial_m g_{jk} \right)$ the Christoffel symbols.  The calculus of variations shows the geodesics are the curves satisfying the following second order differential system:
\[  \ddot{\gamma}^i + \Gamma^i_{jk} \dot{\gamma}^j \dot{\gamma}^k = 0.
\]

 The fundamental theorem of Riemannian geometry states that on any Riemannian manifold there is a unique (torsion-free) connection which is compatible with the metric, called the Levi-Civita (or metric) connection. For that choice of connection, shortest paths (geodesics) are auto-parallel curves ("straight lines"). This connection is determined in a local coordinate system through the  Christoffel symbols:   $\nabla_{\partial_i}\partial_j =  \Gamma_{ij}^k \partial_k$.	With these conventions, the covariant derivative of the coordinates $v^i$ of a vector field is $v^i_{;j} = (\nabla_j v)^i = \partial_j v^i +\Gamma^i_{jk} v^k$.

In the following, we only consider the Levi-Civita connection and we assume that the manifold is geodesically complete, i.e. that the definition domain of all geodesics can be extended to $\R$. This means that the manifold has no boundary nor any singular point that we can reach in a finite time. As an important consequence, the Hopf-Rinow-De~Rham theorem states that there always exists at least one minimizing geodesic between any two points of the manifold (i.e. whose length is the distance between the two points).

\subsection{Normal coordinate systems}
\label{ExpMapIntro}

Let $x$ be a point of the manifold that we consider as a local reference and $v$ a vector of the tangent space $T_{x}\M$ at that point.  From the theory of second order differential equations, we know that there exists one and only one geodesic $\gamma_{(x,v)}(t)$ starting from that point with this tangent vector. This allows to wrap the tangent space onto the manifold, or equivalently to develop the manifold in the tangent space along the geodesics (think of rolling a sphere along its tangent plane at a given point). The mapping $ \exp_{x}(v) = \gamma_{(x,v)}(1)$ of each vector $v \in
T_{x}\M$ to the point of the manifold that is reached after a unit time by the geodesic $\gamma_{(x,v)}(t)$  is called the {\em exponential map} at point $x$. Straight lines going through 0 in the tangent space are transformed into geodesics going through point $x$ on the manifold and distances along these lines are conserved.

The exponential map is defined in the whole tangent space $T_{x}\M$ (since the manifold is geodesically complete) but it is generally one-to-one only locally around 0 in the tangent space (i.e.  around $x$ in the manifold). In the sequel, we denote by $\lcp{xy}=\log_{x}(y)$ the inverse of the exponential map: this is the smallest vector (in norm) such that $y = \exp_{x}(\lcp{xy})$. It is natural to search for the maximal domain where the exponential map is a diffeomorphism. If we follow a geodesic  $\gamma_{(x, v)}(t) = \exp_{x}(t\: v)$ from $t=0$ to infinity, it is either always minimizing all along or it is minimizing up to a time $t_0 < \infty$ and not any more after (thanks to the geodesic completeness). In this last case, the point $ \gamma_{(x,v)}(t_0)$ is called a {\em cut point} and the corresponding tangent vector $t_0\: v$ a {\em tangential cut point}. The set of tangential cut points at $x$ is called the {\em tangential cut locus} $C(x) \in T_{x}\M$, and the set of cut points of the geodesics starting from $x$ is the {\em cut locus} $\C(x) = \exp_{x}(C(x)) \in \M$. This is the closure of the set of points where several minimizing geodesics starting from $x$ meet. On the sphere ${\mathcal S}_2(1)$ for instance, the cut locus of a point $x$ is its antipodal point and the tangential cut locus is the circle of radius $\pi$.	

The maximal bijective domain of the exponential chart is the domain  $D(x)$ containing 0 and delimited by the tangential cut locus ($\partial D(x) = C(x)$). This domain is connected and star-shaped with respect to the origin of  $T_{x}\M$. Its image by the exponential map covers all the manifold except the cut locus, which has a null measure.  Moreover, the segment $[0,\lcp{xy}]$ is mapped to the unique minimizing geodesic from $x$ to $y$: geodesics starting from $x$ are straight lines, and the distance from the reference point are conserved. This chart is somehow the ``most linear'' chart of the manifold with respect to the reference point $x$. 

When the tangent space is provided with an orthonormal basis, this is called {\em an normal coordinate systems at $x$}. A set of normal coordinate systems at each point of the manifold realize an atlas which allows to work very easily on the manifold. The  implementation of the exponential and logarithmic maps (from now on $\exp$ and $\log$) is indeed the basis of programming on Riemannian manifolds, and we can express using them practically all the geometric operations needed for statistics \citep{A:pennec:inria-00614994} or image processing \citep{A:pennec:inria-00614990}.

The size of the maximal definition domain is quantified by the {\em   injectivity radius} $\mbox{inj}(\M,x) = \dist(x,\C(x))$, which is the maximal radius of centered balls in $T_{x}\M$ on which the exponential map is one-to-one. The injectivity radius of the manifold $\mbox{inj}(\M)$ is the infimum of the injectivity over the manifold. It may be zero, in which case the manifold somehow tends towards a singularity (think e.g. to the surface $z=1/\sqrt{x^2+y^2}$ as a sub-manifold of $\R^3$).

In a Euclidean space, normal coordinate systems are realized by orthonormal coordinates system translated at each point: we have in this case $\lcp{xy} = \log_{x}(y) = y-x$ and $\exp_{x}(\lcp{v}) = x+\lcp{v}$. This example is more than a simple coincidence. In fact, most of the usual operations using additions and subtractions may be reinterpreted in a Riemannian framework using the notion of {\em bipoint}, an antecedent of vector introduced during the 19th Century. Indeed, vectors are defined as equivalent classes of bipoints in a Euclidean space. This is possible because we have a canonical way (the translation) to compare what happens at two different points. In a Riemannian manifold, we can still compare things locally (by parallel transportation), but not any more globally. This means that each ``vector'' has to remember at which point of the manifold it is attached, which comes back to a bipoint.

\section{Hessian of the squared distance}

\subsection{Computing the differential of the Riemannian log}
On $\M / C(y)$, the Riemannian gradient $\nabla^a = g^{ab} \partial_b$ of the squared distance $d^2_y(x)=\dist^2(x, y)$ with respect to the fixed point $y$ is well defined and is equal to $\nabla d^2_y(x) = -2 \log_x(y)$.  The Hessian operator (or double covariant derivative) $\nabla^2 f(x)$ from $T_x\M$ to $T_x\M$ is the covariant derivative of the gradient, defined by the identity $\nabla^2 f(v) = \nabla_v(\nabla f)$. In  a normal coordinate system at point $x$, the Christoffel symbols vanish at $x$, so that the Hessian operator of the squared distance can be expressed with the standard differential $D_x$ with respect to the point $x$:
\[
\nabla^2 d^2_y(x) = -2 (D_x \log_x(y)).
\]
The points $x$ and $y=\exp_x(v)$ are called conjugate if $D\exp_x(v)$ is singular. It is known that the cut point (if it exists) occurs at or before the first conjugate point along any geodesic \citep{A:LeeCurvature:1997}. Thus, $D\exp_x(v)$ has full rank inside the tangential cut-locus of $x$. This is in essence why there is a well posed inverse function  $\lcp{x y} = \log_x(y)$,  called the Riemannian log, which is continuous and differentiable everywhere except at the cut locus of $x$. Moreover, its differential can be computed easily: since $\exp_x(\log_x(y)) =y $, we have $\left. D\exp_x \right|_{\lcp{xy}}  D\log_x (y) = \Id$, so that 
\begin{equation}
D\log_x (y) =  \left( \left. D\exp_x \right|_{\lcp{xy}} \right)^{-1}
\label{eq:Dylogxy}
\end{equation}
is well defined and of full rank on $\M/C(x)$.

We can also see the Riemannian log $\log_x(y) = \lcp{x y}$ as a function of the foot-point $x$, and differentiating $\exp_x(\log_x(y))=y$ with respect to it gives:
$
\left. D_x \exp_x \right|_{\lcp{xy}} + \left. D\exp_x \right|_{\lcp{xy}}.D_x\log_x (y) =0.
$ 
Once again, we obtain a well defined and full rank differential for $x \in \M/C(y)$:
\begin{equation}
D_x\log_x (y) = - \left( \left. D\exp_x \right|_{\lcp{xy}}\right)^{-1} \left. D_x \exp_x \right|_{\lcp{xy}}.
\label{eq:Dxlogxy}
\end{equation}
The Hessian of the squared distance can thus be written:
\[
\frac{1}{2}\nabla^2 d^2_y(x) = - D_x \log_x(x_i) =  \left( \left. D\exp_x \right|_{\lcp{xy}}\right)^{-1} \left. D_x \exp_x \right|_{\lcp{xy}}.
\]
If we notice that $J_0(t) = \left. D\exp_x\right|_{t \lcp{xy}}$ (respectively $J_1(t) = \left. D_x \exp_x\right|_{t \lcp{xy}}$) are actually matrix Jacobi field solutions of the Jacobi equation $\ddot J (t) + R(t) J(t) =0$ with $J_0(0)=0$ and $\dot J_0(0)=\Id_n$ (respectively $J_1(0)=\Id_n$ and $\dot J_1(0)=0$), we see that the above formulation of the Hessian operator is equivalent to the one of \cite{A:villani_regularity_2011}[Equation 4.2]: $\frac{1}{2}\nabla^2 d^2_y(x) = J_0(1)\inv J_1(1)$.

\subsection{Taylor expansion of the Riemannian log}

In order to better figure out what is the dependence of the Hessian of the squared Riemannian distance with respect to curvature, we compute here the Taylor expansion of the Riemannian  log function. Following \cite{A:brewin_riemann_2009}, we consider a normal coordinate system centered at $x$ and $x_v = \exp_x(v)$ a variation of the point $x$. We denote by $R_{ihjk}(x)$ the coefficients of the curvature tensor at $x$ and by $\epsilon$ a conformal gauge scale that encodes the size of the path  in terms of $\|v \|_x$ and $\| \lcp{xy} \|_x$ normalized by the curvature (see \cite{A:brewin_riemann_2009} for details).

In a normal coordinate system centered at $x$,   we have the following Taylor expansion of the metric tensor coefficients:
\begin{equation}
\begin{split}
g_{ab}(v) = & g_{ab} - \frac{1}{3} R_{cabd}v^c v^d 
- \frac{1}{6} \nabla_e R_{cabd} v^e v^c v^d 
\\ & + \left( - \frac{1}{20} \nabla_e \nabla_f R_{cabd} + \frac{2}{45} R_{cad}^g R_{ebf}^h \delta_{gh} \right) v^c v^d v^e v^f + O(\epsilon^5). %
\end{split}
\label{eq:TaylorMetric}
\end{equation}  

 A geodesic joining point $z$ to point $z+\delta z$ has tangent vector: 
\begin{eqnarray*}
\left[ \log_z(z+\Delta z) \right]^a 
&= &\Delta z^a  +\frac{1}{3}  z^b  \Delta z^c \Delta z^d R^a_{cbd}
              + \frac{1}{12} z^b z^c \Delta z^d \Delta z^e \nabla_d R^a_{bce}
\\ && + \frac{1}{6} z^b z^c \Delta z^d \Delta z^e \nabla_b R^a_{dce}
      + \frac{1}{24} z^b z^c \Delta z^d \Delta z^e \nabla^a R_{bdce}
\\ && + \frac{1}{12}  z^b \Delta z^c \Delta z^d \Delta z^e \nabla_c R^a_{dbe} 
+ O(\epsilon^4).
\end{eqnarray*}

Using $ z= v$ and $z+\Delta z = \lcp{xy}$ (i.e. $\Delta z = \lcp{xy} -v)$ in a normal coordinate system centered at $x$, and keeping only the first order terms in $v$, we obtain the first terms of the series development of  the log:
\begin{equation}
\label{eq:TaylorLog}
\begin{split}
 \left[ \log_{x +v}(y) \right]^a  
& =  \lcp{xy}^a -v^a + \frac{1}{3}  R^a_{cbd} v^b \lcp{xy}^c  \lcp{xy}^d 
 + \frac{1}{12} \nabla_c R^a_{dbe}  v^b \lcp{xy}^c \lcp{xy}^d \lcp{xy}^e 
+ O(\epsilon^4).
\end{split}
\end{equation}
Thus, the differential of the log with respect to the foot point is:
\begin{equation}
\label{eq:Diff_logSupp} 
- \left[  D_x \log_x(y) \right]^a_b = \delta^a_b - \frac{1}{3} R^a_{cbd} \lcp{xy}^c \lcp{xy}^d - \frac{1}{12} \nabla_c R^a_{dbe}  \lcp{xy}^c \lcp{xy}^d \lcp{xy}^e  + O(\epsilon^3).
\end{equation}
Since we are in a normal coordinate system,  the zeroth order term is the identity matrix, like in the Euclidean space, and the first order term vanishes. The Riemannian curvature tensor appear in the second order term and its covariant derivative in the third order term. The important point here is to see that the curvature is the leading term that makes this matrix departing from the identity (i.e. the Euclidean case) and which may lead to the non invertibility of the differential.

\section{Example on spheres}
\label{sec:sphere}

We consider the unit sphere in dimension $n \geq 2$ embedded in $\R^{n+1}$ and we represent points of $\M = {\cal S}_n$ as unit vectors in $\R^{n+1}$. The tangent space at $x$ is naturally represented by the linear space of vectors orthogonal to $x$:  $T_x{\cal S}_n = \{ v \in \R^{n+1}, v\trp x =0\}$. The natural Riemannian metric on the unit sphere is inherited from the Euclidean metric of the embedding space $\R^{n+1}$. With these conventions, the Riemannian distance is the arc-length $d(x,y) = \arccos( x\trp y)= \theta \in [0,\pi]$. Denoting $f(\theta) = 1/ \mbox{sinc}(\theta) = { \theta}/{\sin(\theta)}$, the spherical exp and log maps are:
\begin{eqnarray}
\exp_x(v) & = & \cos(\| v\|) x +  \mbox{sinc}(\| v\|) v / \| v\| \\
\log_x(y) & = & f(\theta) \left( y - \cos(\theta) x \right)
\quad \text{with} \quad \theta = \arccos(x\trp y).
\end{eqnarray}
Notice that  $f(\theta)$  is a smooth function from $]-\pi;\pi[$ to $\R$ that is always greater than one and is locally quadratic at zero: $f(\theta) = 1 +\theta^2/6 + O(\theta^4)$.

\subsection{Hessian of the squared distance on the sphere}

To compute the gradient and Hessian of functions on the sphere, we first need a chart in a neighborhood of a point $x\in {\cal S}_n$. We consider  the unit vector $x_v = \exp_x(v)$ which is a variation of $x$ parametrized by the tangent vector $v \in T_x{\cal S}_n$ (i.e. verifying $x\trp v=0$). In order to extend this mapping to the embedding space to simplify computations, we consider that $v$ is the orthogonal projection of an unconstrained vector $w \in \R^{n+1}$ onto the tangent space at $x$: $v=(\Id -x x\trp)w$. Using the above formula for the exponential map, we get at first order $x_v  =   x  -  v + O(\|v\|^2)$ in the tangent space or $x_w = x + (\Id -x x\trp)w +  O(\|w\|^2)$ in the embedding space.

It is worth verifying first that the gradient of the squared distance  $\theta^2 = d^2_y(x) =  \arccos^2\left( {x\trp y} \right)$ is indeed $\nabla d^2_y(x) = -2 \log_x(y)$. We considering the variation $x_w = \exp_x( (\Id-xx\trp)w)= x +(\Id-xx\trp)w + O(\|w\|^2)$. Because $D_x \arccos(y \trp x) = -y \trp / \sqrt{ 1 - (y \trp x)^2}$, we get:
\[
D_w \arccos^2\left( {x_w\trp y} \right)
= \frac{ -2 \theta}{\sin \theta} y\trp (\Id-xx\trp)
= -2 f(\theta) y\trp (\Id-xx\trp),
\]
and the gradient is as expected:
\begin{equation}
 \nabla d^2_y(x) = -2 f(\theta) (\Id -xx\trp)y = -2 \log_x(y).
\label{eq:GradDistSphere}
\end{equation}

To obtain the Hessian, we now compute the Taylor expansion of $\log_{x_w}(y)$. First, we have 
\[
f(\theta_w) 
= f(\theta) - \frac{f'(\theta)}{\sin \theta} {y\trp (\Id-xx\trp)w} + O(\|w\|^2),
\]
with  $ f'( \theta ) 
= (1-f(\theta)\cos \theta)/\sin \theta$.
Thus, the first order Taylor expansion of $\log_{x_w}(y) = f(\theta_w) ( y - \cos(\theta_w) x_w )$ is:
\[
\begin{split}
\log_{x_w}(y) 
& = 
f(\theta_w)
\left( \Id -xx\trp -(\Id  -xx\trp)w x\trp - x w\trp (\Id  -xx\trp) \right)y  + O(\|w\|^2) \\
\end{split}
\]
so that
\[
\begin{split}
 -2 D_w \log_{x_w}(y) = 
\frac{f'(\theta)}{\sin \theta}  ( \Id -xx\trp )y y\trp (\Id-xx\trp)  - f(\theta) \left( x\trp y \Id    + x y\trp \right) ( \Id -xx\trp ) 
\end{split}
\]
Now, since we have computed the derivative in the embedding space, we have obtained the Hessian with respect to the flat connection of the embedding space, which exhibits a non-zero normal component. In order to obtain the Hessian with respect to the connection of the sphere, we need to project back on $T_x{\cal S}_n$ (i.e. multiply by $(\Id -x x\trp)$ on the left) and we obtain:
\[
\begin{split}
\frac{1}{2} H_x(y) 
& =  \left( \frac{1- f(\theta) \cos\theta}{\sin^2 \theta }  \right) \left( \Id - x x\trp  \right) yy\trp (\Id -xx\trp) 
 +  f( \theta  )\cos \theta (\Id -xx\trp)   \\
& = \left( \Id - x x\trp  \right) \left( ( 1 - f(\theta) \cos\theta ) \frac{ yy\trp}{ \sin^2\theta} 
 +  f( \theta  )\cos\theta \Id \right) (\Id -xx\trp).
\end{split}
\]

To simplify this expression, we note that $\|(\Id-xx\trp)y\|^2 = \sin \theta$, so that $u = \frac{(\Id-xx\trp)y }{ \sin \theta} = \frac{ \log_x(y) }{\theta}$ is a unit vector of the tangent space at $x$ (for $y \not = x$ so that $\theta > 0$). Using this unit vector and the intrinsic parameters $\log_x(y)$ and $\theta = \| \log_x(y)\|$,  we can rewrite the Hessian:
\begin{eqnarray}
\qquad \frac{1}{2} H_x(y) & =  & f( \theta  )\cos\theta  (\Id -xx\trp) 
+  \left( \frac{ 1- f(\theta)\cos\theta}{\theta^2 } \right) \log_x(y) \log_x(y)\trp
\\
& =  & u u\trp + f( \theta  )\cos\theta  (\Id -xx\trp - u u\trp) 
\end{eqnarray}

The eigenvectors and eigenvalues of this matrix are now very easy to determine. By construction, $x$ is an eigenvector with eigenvalue $\mu_0=0$. Then the vector $u$ (or equivalently $\log_x(y) = f(\theta) (\Id-x x \trp) y = \theta u$) is an eigenvector with eigenvalue $\mu_1=1$. Finally, every vector $u$ which is orthogonal to these two vectors (i.e. orthogonal to the plane spanned by 0, $x$ and $y$) has eigenvalue $\mu_2= f(\theta)\cos\theta = \theta \cot \theta$. This last eigenvalue is positive for $\theta \in [0,\pi/2[$, vanishes for $\theta = \pi/2$ and becomes negative for $\theta \in ]\pi/2 \pi[$.
We retrieve here the results of \cite[lemma 2]{A:buss_spherical_2001} expressed in a more general coordinate system.

\section{Example on the hyperbolic space $\Hyp^n$}

We consider in this section the hyperboloid of equation $-x_0^2 + x_1^2 \ldots  x_n^2 = -1$ (with $x_0 > 0$ and $n \geq 2$) embedded in $\R^{n+1}$. Using the notations $x=(x_0,\hat x)$ and the indefinite nondegenerate symmetric bilinear form $\scal{x}{y}_* = x\trp J y= \hat x\trp \hat y -x_0 y_0$ with $ J = \mbox{diag}(-1, \Id_n)$, the hyperbolic space  can be seen as the sphere $\|x\|^2_* =-1$ of radius -1 in the $(n+1)$-dimensional Minkowski space:
\[
\Hyp^n = \{ x \in \R^{n,1} / \|x\|^2_* = \|\hat x\|^2 -x_0^2 = -1 \}.
\]
A point in $\M = \Hyp^n \subset \R^{n,1}$ can be parametrized  by $x=(\sqrt{1+\|\hat x\|^2}, \hat x)$ for $\hat x \in \R^n$ (Weierstrass coordinates). This happen to be in fact a global diffeomorphism that provides a very convenient global chart of the hyperbolic space. We denote $\pi(x)=\hat x$ (resp. $\pi\inv(\hat x)= (\sqrt{1+\|\hat x\|^2}, \hat x)$) the coordinate map from $\Hyp^n$ to $\R^n$ (resp. the parametrization map from $\R^n$ to $\Hyp^n$). The Poincarr\'e ball model is another classical models of the hyperbolic space $\Hyp^n$ which can be obtained by a stereographic projection of the hyperboloid onto the hyperplane $x_0 = 0$ from the south pole $(-1, 0 \ldots, 0)$. 

A tangent vector $v=(v_0, \hat v)$  at point $x=(x_0,\hat x)$ satisfies $\scal{x}{v}_* = 0$, i.e. $x_0 v_0 = \hat x\trp \hat v$, so that 
\[
T_x \Hyp^n = \left\{  \left( \frac{\hat x\trp \hat v}{\sqrt{1+\|\hat x\|^2}}, \hat v\right),\quad \hat v\in \R^{n} \right\}.
\]
The natural Riemannian metric on the hyperbolic space is inherited from the Minkowski metric of the embedding space $\R^{n,1}$: the scalar product of two vectors $u=(\hat x\trp \hat u / \sqrt{1+\|\hat x\|^2},\hat u)$ and $v=(\hat x\trp \hat v / \sqrt{1+\| \hat x\|^2}, \hat v)$ at $x=(\sqrt{1+\|\hat x\|^2}, \hat x)$ is 
\[
\scal{u}{v}_* = u\trp J v = -u_0 v_0 + \hat u\trp \hat v = \hat u\trp \left( -\frac{\hat x \hat x\trp}{1+\|\hat x\|^2} + \Id \right)  \hat v 
\] 
The metric matrix expressed in the coordinate chart $G=\Id - \frac{ \hat x \hat x\trp}{1+\|\hat x\|^2}$ has eigenvalue 1, with multiplicity $n-1$,  and $1/(1+\|\hat x\|^2)$ along the eigenvector $x$. It is thus positive definite. 

With these conventions, geodesics are the trace of 2-planes passing through the origin and the Riemannian distance is the arc-length:
\begin{equation} 
   d(x,y) = \arccosh( - \scal{x}{y}_* ). 
\end{equation}
The hyperbolic exp and log maps are:
\begin{eqnarray}
\quad \exp_x(v) &=& \cosh(\| v\|_* ) x +  {\sinh(\| v\|_* )} v / {\| v\|_* }
\\
\log_x(y) &=& f_*(\theta) \left( y - \cosh(\theta) x \right)
\quad \text{with} \quad \theta = \arccosh( -\scal{x}{y}_* ),
\end{eqnarray}
where $f_*(\theta) = { \theta}/{\sinh(\theta)}$ is a smooth function from $\R$ to $(0,1]$ that is always positive and is locally quadratic at zero: $f_*(\theta) = 1 - \theta^2/6 + O(\theta^4)$.

\subsection{Hessian of the squared distance on the hyperbolic space}

We first verify that the gradient of the squared distance $d^2_y(x) =  \arccosh^2\left( -<x,y>_* \right)$ is indeed $\nabla d^2_y(x) = -2 \log_x(y)$.  Let us consider a variation of the base-point along the tangent vector $v$ at $x$ verifying $\scal{v}{x}_*=0$: 
\[ x_{v} = \exp_x(v) = \cosh(\| v\|_* ) x +  \frac{\sinh( \| v\|_* )}{ \| v\|_* } v
 = x +   v  + O( \| v \|_*^2).
\]
In order to extend this mapping to the embedding space around the paraboloid, we consider that $v$ is the projection $v=w + \scal{w}{x}_* x$ of an unconstrained vector $w\in \R^{n,1}$ onto the tangent space at $T_x \Hyp^n$. Thus, the variation that we consider in the embedding space is
\[
x_w = x + \partial_w x_w + O(\|w\|^2_Q) \quad \mbox{with} \quad 
\partial_w x_w = w + \scal{w}{x}_* x = (\Id + x x\trp J) w.
\]

Now, we are interested in the impact of such a variation on $\theta_w = d_y(x_w) =\arccosh\left( - \scal{x_w}{y}_* \right)$. Since $\arccosh'(t) = \frac{1}{\sqrt{t^2 -1}}$, and $\sqrt{\cosh(\theta)^2 -1} = \sinh(\theta)$ for a positive $\theta$, we have:
\[
{d}/{dt} \left. \arccosh(t) \right|_{t=\cosh(\theta)} = 
 { 1}/{{\sqrt{\cosh(\theta)^2 -1}}}
= {1}/{\sinh(\theta)},
\]
so that 
\[
\theta_w = \theta -  \frac{1}{\sinh(\theta)} \scal{w + \scal{w}{x}_* x}{y}_* + O(\|v\|_*^2).
\]
This means that the directional derivative is 
\[
\partial_w \theta_w =  -  \frac{1}{\sinh(\theta)} \scal{w + \scal{w}{x}_* x}{y}_* 
= - \frac{1}{\sinh(\theta)} \scal{w}{y -\cosh(\theta)  x} 
\]
so that $ \partial_w \theta_w^2 = -2 f_*(\theta) \scal{w}{y - \cosh(\theta) x }_*.$ Thus, the gradient in the embedding space defined by $<\nabla d^2_y(x) , w>_* = \partial_w \theta_w^2$ is as expected:
\begin{equation}
\nabla d^2_y(x) =  
- 2 f_*(\theta) (y- \cosh(\theta) x) = - 2 \log_x(y).
\end{equation}

To obtain the Hessian, we now compute the Taylor expansion of $\log_{x_w}(y)$. First, we compute the variation of  $f_*(\theta_w) = \theta_w / \sinh(\theta_w)$:
\[
\partial_w f_*(\theta_w) = {f_*'(\theta)} \: \partial_w \theta_w
= - \frac{f_*'(\theta)}{\sinh(\theta)} \scal{w}{y -\cosh(\theta)  x}_*
= - \frac{f_*'(\theta)}{\theta} \scal{w}{\log_x(y)}_*
\]
with  $ f_*'( \theta ) = (1-f_*(	\theta)\cosh \theta)/\sinh \theta = (1- \theta \coth \theta)/\sinh \theta$. The variation of $\cosh \theta_w$ is:
\[
\partial_w \cosh \theta_w = \sinh \theta \: \partial_w \theta_w 
= - \scal{w}{y -\cosh(\theta)  x}_*.
\]
Thus, the first order variation of $\log_{x_w}(y)$ is:
\[
\begin{split}\partial_w \log_{x_w}(y) 
&= \partial_w f_*(\theta_w) (y-\cosh \theta x ) - f_*(\theta) \left( \partial_w \cosh(\theta_w) x + \cosh(\theta) \partial_w x_w \right) \\
&= - \frac{f_*'(\theta)\sinh\theta}{\theta^2} \scal{w}{\log_x(y)}_*  \log_x(y) \\
	&\:\: + f_*(\theta) \left( \scal{w}{y -\cosh(\theta)  x}_* x  -\cosh(\theta) (w + \scal{w}{x}_* x)\right) \\
	&= - \frac{(1- \theta \coth \theta)}{\theta^2} \scal{w}{\log_x(y)}_*  \log_x(y) \\
	&\:\: + \scal{w}{\log_x(y)}_* x  - \theta \coth(\theta) (w + \scal{w}{x}_* x).
\end{split}
\]
This vector is a variation in the embedding space: it displays a normal component to the hyperboloid $ \scal{w}{\log_x(y)}_* x $ which reflects the extrinsic curvature of the hyperboloid in the Minkowski space (the mean curvature vector is $-x$), and a tangential component which measures the real variation in the tangent space:
\[
\begin{split}
(\Id + x x\trp J) \partial_w \log_{x_w}(y)  = 
& - \frac{(1- \theta \coth \theta)}{\theta^2} \scal{w}{\log_x(y)}_*  \log_x(y) \\
& - \theta \coth(\theta) (J + x x\trp) J w.
\end{split}
\]
Thus the intrinsic gradient is:
\[
D_x \log_{x}(y)  =  
- \frac{(1- \theta \coth \theta)}{\theta^2}  \log_x(y) \log_x(y)\trp J
- \theta \coth(\theta) (\Id + x x\trp J).
\]
Finally, the Hessian of the square distance, considered as an operator from $T_x\Hyp^n$ to $T_x\Hyp^n$,  is $H_x(y)(w) = -2 D_x \log_{x}(y) w$.  Denoting $u=  \log_x(y) / \theta$ the unit vector of the tangent space at $x$ pointing towards the point $y$, we get in matrix form: 
\[
\frac{1}{2} H_x(y) 
= u u\trp J + \theta \coth \theta  (J + x x\trp -u u\trp) J 
\]
In order to see that the Hessian is symmetric, we have to lower an index (i.e. multiply on the left by J) to obtain the bilinear form:
\[
 H_x(y) (v,w) = \scal{v}{H_x(y)(w)}_* = 2 v\trp J \left( u u\trp + \theta \coth \theta (J + x x\trp -u u\trp) \right) J w.
\]

The eigenvectors and eigenvalues of (half) the Hessian operator are now easy to determine. By construction, $x$ is an eigenvector with eigenvalue $0$ (restriction to the tangent space). Then, within the tangent space at $x$, the vector $u$ (or equivalently $\log_x(y) = \theta u$) is an eigenvector with eigenvalue $1$. Finally, every vector $v$ which is orthogonal to these two vectors (i.e. orthogonal to the plane spanned by 0, $x$ and $y$) has eigenvalue $\theta \coth \theta \geq 1$ (with equality only for $\theta=0$). Thus, we can conclude that the Hessian of the squared distance is always positive definite and does never vanish along the hyperbolic space. This was of course expected since it is well known that the Hessian stay positive definite for negatively curved spaces \citep{A:bishop_manifolds_1969}. As a consequence, the squared distance is a convex function and has a unique minimum. 


\clearpage

\setcounter{equation}{0}
\setcounter{figure}{0}
\setcounter{table}{0}
\setcounter{page}{1}
\setcounter{section}{0}
\setcounter{theorem}{0}
\setcounter{definition}{0}
\setcounter{proposition}{0}
\makeatletter
\renewcommand{\thetheorem}{B\arabic{theorem}}
\renewcommand{\thepage}{B\arabic{page}}
\renewcommand{\thesection}{B\arabic{section}}
\renewcommand{\theequation}{B\arabic{equation}}
\renewcommand{\thefigure}{\arabic{figure}}
\renewcommand{\bibnumfmt}[1]{[B#1]}
\renewcommand{\citenumfont}[1]{B#1}

\begin{frontmatter}
\pdfbookmark[0]{Supplementary Materials B: Euclidean PCA as an optimization in the flag space}{SupB}

\title{Supplementary Materials B: Euclidean PCA as an optimization in the flag space}
\runtitle{Euclidean PCA as an optimization in the flag space}



\author{\fnms{Xavier} \snm{Pennec}\corref{}\ead[label=e1B]{xavier.pennec@inria.fr}}
\address{Asclepios team, Inria Sophia-Antipolis M\'editerran\'ee \\ 2004 Route des Lucioles, BP93
\\ F-06902 Sophia-Antipolis Cedex, France\\ \printead{e1B}}
\affiliation{Universit\'e C\^ote d'Azur and Inria, France}

\runauthor{X. Pennec}


\begin{abstract}
This supplementary material details in length the proof that the flag of linear subspaces found by PCA optimizes the Accumulated Unexplained Variances (AUV) criterion in a Euclidean space.
\end{abstract}

\end{frontmatter}

\section{A QR decomposition of the reference matrix}

 Let $X=[x_0, \ldots x_k]$ be a matrix of $k+1$ independent reference points in $\R^n$. Following the notations of the main paper, we write the reference matrix 
\[
Z(x) = [x-x_0,  \ldots x-x_k] = x\mathds{1}_{k+1} \trp - X.
\]
 The affine span $\Aff(X)$ is the locus of points $x$ satisfying $Z(x)\lambda = 0$ 
i.e. $x = X \lambda / (\mathds{1}_{k+1}\trp \lambda)$. Here, working with the barycentric weights is not so convenient, and in view of the principal component analysis, we prefer to work with a variant of the QR decomposition using the Gram-Schmidt orthogonalization process.

Choosing $x_0$ as the pivot point, we iteratively decompose $X - x_0 \one_{k+1}\trp$ to find an orthonormal basis of the affine span of $X$. For convenience, we define the zeroth vectors $v_0= q_0 =0$. The first axis is defined by $v_1 = x_1-x_0$, or by the unit vector $q_1 = v_1 / \| v_1\|$. Next, we  project the second direction $x_2-x_0$ onto $\Aff(x_0, x_1) = Aff(x_0, x_0 + e_1)$:  the orthogonal component $v_2 = (\Id - e_1 e_1\trp) (x_2 -x_0)$ is described by the unit vector $q_2 = v_2 / \| v_2\|$. The general iteration is then (for $i\geq 1$):
\[
v_i = (\Id - \sum_{j=0}^{i-1} e_j e_j\trp) (x_i - x_0),  \qquad \text{and} \qquad q_i = v_i / \| v_i\|.
\]
Thus, we obtain the decomposition:
\[
\begin{split}
X & = x_0 \one_{k+1}\trp + Q T \\
Q & = [q_0,  q_1,  \ldots q_k] \\
T & = \left[ 
\begin{array}{ccccc}
q_0\trp (x_0 -x_0) & q_0\trp (x_1 -x_0) & q_0\trp (x_2 -x_0) & \ldots & q_0\trp (x_k -x_0) \\
0 & q_1\trp (x_1 -x_0) & q_1\trp (x_2 -x_0) & \ldots & q_1\trp (x_k -x_0) \\
0 &                 0  & q_2\trp (x_2 -x_0) & \ldots & q_2\trp (x_k -x_0) \\
0 &                 0  &             \ldots & \ldots &             \ldots \\
0 &                 0  &             \ldots & \ldots & q_{k}\trp (x_k -x_0) 
\end{array}
\right]
\end{split}
\]
With this affine variant of the QR decomposition, the $(k+1)\times (k+1)$ matrix $T$ is triangular superior with vanishing first row and first column (since $q_0=0$). The  $n\times (k+1)$ matrix $Q$ also has a first null vector before the usual $k$ orthonormal vectors in its $k+1$ columns. The decomposition into matrices of this form is unique when we assume that all the points $x_0, \ldots x_k$ are linearly independent. This means that we can parametrize the matrix $X$ by the orthogonal  (aside the first vanishing column) matrix $Q$ and the triangular (with first row and column zero matrix) $T$.

In view of PCA, it is important to notice that the decomposition is stable under the addition/removal of reference points. Let $X_i=[x_0, \ldots x_{i}]$ be the matrix of the first $i+1$ reference points (we assume $i<k$ to simplify here) and $X_i = x_0 \one_{i+1}\trp + Q_i T_i$ its QR factorization.  Then, the matrix $Q_i$ is made of the first $i+1$ columns of $Q$  and the matrix $T_i$ is the upper $(i+1) \times (i+1)$ bloc of the upper triangular matrices $T$.

\section{Optimizing the $k$-dimensional subspace}

With our decomposition, we can now write any point of $x \in \Aff(X)$ as the base-point $x_0$ plus any linear combination of the vectors $q_i$: $ x = x_0 + Q \alpha$ with $\alpha \in \R^{k+1}$. 
The projection of a point $y$ on $\Aff(X)$ is thus parametrized by the $k+1$ dimensional vector $\alpha$ that minimizes the (squared) distance $d(x,y)^2 = \| x_0 + Q \alpha -y \|^2$. 
Notice that we have $Q\trp Q = \Id_{k+1} -e_1 e_1\trp$ (here $e_1$ is the first basis vector of the embedding space $\R^{K+1}$) so that $Q^{\dag} = Q\trp$.  
The null gradient of this criterion implies that $\alpha$ is solving 
$Q\trp Q \alpha = Q\trp (y-x_0)$, i.e. $\alpha = Q^{\dag} (y-x_0) = Q\trp(y-x_0) $.
Thus, the projection of $y$ on $\Aff(X)$ is 
\[
Proj(y, \Aff(X)) = x_0 + Q Q\trp (y-x_0),
\]
 and the residue is
\[
\begin{split}
r^2(y) & = \| (\Id_{n} - Q Q\trp) (y-x_0)\|^2 = \tr\left( (\Id_{n} - Q Q\trp) (y-x_0)(y-x_0)\trp \right).
\end{split}
\]

Accounting now for the $N$ data points ${ Y} = \{ y_i \}_{i=1}^N$, and denoting as usual $\bar y = \frac{1}{N} \sum_{i=1}^N   y_i$ and $\Sigma = \frac{1}{N} \sum_{i=1}^N  ( y_i -\bar y) ( y_i -\bar y)\trp$, 
the unexplained variance is:
\[
\sigma_{out}^2(X) 
= \tr\left(  (\Id_{n} - Q Q\trp) ( \Sigma - (\bar y-x_0)(\bar y-x_0)\trp ) \right)
.
\]
In this formula, we see that the value of the upper triangular matrix $T$ does not appear and can thus be chosen freely. The point $x_0$ that minimizes the unexplained variance is evidently  $x_0 = \bar y$. To determine the matrix $Q$, we diagonalize the empirical covariance matrix to obtain the spectral decomposition $\Sigma = \sum_{j=1}^n \sigma_j^2 u_j u_j\trp$ where by convention, the eigenvalues are sorted in decreasing order. 
The remaining unexplained variance  $\sigma_{out}^2(X) 
= \tr\left(  (\Id_{n} - (U\trp Q) (U\trp Q)\trp) \mbox{Diag}(\sigma_i^2) \right)$ reaches its minimal value $ \sum_{i=k+1}^n \sigma_i^2$ for $[q_1, \ldots q_k] = [u_1, \ldots u_k] R$ where $R$ is any $k\times k$ orthogonal matrix. Here, we see that the solution is unique in terms of subspaces (we have $\text{Span}(q_1, \ldots q_k) = \text{Span}(u_1, \ldots u_k)$ whatever orthogonal matrix $R$ we choose) but not in terms of the matrix $Q$. 
In particular, the matrix $X = [\bar y, \bar y + u_1,\ldots \bar y + u_k ]$ is one of the matrices describing the optimal subspace but the order of the vectors is not prescribed.

\section{The AUV criterion}

In PCA, one often plots the unexplained variance as a function of the number of modes used to approximate the data. This curve should decreases as fast as possible from the variance of the data (for 0 modes) to 0 (for $n$ modes). A standard way to quantify the decrease consists in summing the values at all steps. We show in this section that the optimal flag of subspaces (up to dimension $k$) that optimize  this Accumulated Unexplained Variances (AUV) criterion is precisely the result of the PCA analysis.  

As previously, we consider $k+1$ points $x_i$ but they are now ordered.  We denote by $X_i=[x_0, \ldots x_i]$ the matrix of the first $i+1$ columns of $X=[x_0, \ldots x_k]$. The flag generated by $X$ is thus 
\[
Aff(X_0)=\{x_0\} \subset \ldots \subset Aff(X_i) \subset \ldots \subset Aff(X) \subset \R^n.
\]
The QR decomposition of $X$ gives $k$ orthonormal unit vectors $q_1$ \ldots $q_k$ which can be  complemented by $n-k$  unit vector $q_{k+1}, \ldots q_n$ to constitute an orthonormal basis of $\R^n$. Using this extended basis, we can write:
\[
\sigma_{out}^2(X) 
= \tr\left(  W ( \Sigma - (\bar y-x_0)(\bar y-x_0)\trp ) \right)
\]
with $W= (\Id_{n} - Q Q\trp) = \sum_{j=k+1}^n q_j q_j\trp.$
Since the decomposition is stable under the removal of reference points, the QR factorization of $X_i$ is $X_i = x_0 \one_{i+1}\trp + Q_i T_i$ with $Q_i=[q_0, \ldots q_i]$ and we can write the unexplained variance for the subspace $Aff(X_i)$ as:
\[
\sigma_{out}^2(X_i) 
= \tr\left(  W_i ( \Sigma - (\bar y-x_0)(\bar y-x_0)\trp ) \right)
\]
with $W_i= (\Id_{n} - Q_i Q_i\trp) = \sum_{j=i+1}^n q_j q_j\trp.$
Plugging this value into the criterion 
$AUV(X) = \sum_{i=0}^k \sigma^2_{out}( X_i )$, we get:
\[
AUV(X_k) = \tr\left(  \bar W ( \Sigma - (\bar y-x_0)(\bar y-x_0)\trp ) \right) 
\]
with 
\[
\bar W  = \sum_{i=0}^k W_i
= \sum_{i=0}^k (\Id_{n} - Q_i Q_i\trp) 
=  \sum_{i=0}^k \sum_{j=i+1}^n q_j q_j\trp 
 = \sum_{i=1}^k i q_i q_i\trp + (k+1)  \sum_{i=k+1}^n q_i q_i\trp.
\]

\section{PCA optimizes the AUV criterion}
The minimum over $x_0$ is achieved as before for $x_0= \bar y$ and the AUV for this value it now parametrized only by the matrix $Q$:
\[
AUV(Q) = \tr\left(  U\trp W_k U \mbox{Diag}(\sigma_i^2) \right)
= \sum_{i=1}^k i q_i\trp \Sigma q_i + (k+1)  \sum_{i=k+1}^n  q_i\trp \Sigma q_i.
\]
 Assuming that the first $k+1$ eigenvalues $\sigma_i^2$ ($1\leq i \leq k+1$) of $\Sigma$ are all different (so that they can be sorted in a strict order),  we claim that the optimal unit orthogonal vectors are $q_i = u_i$ for $1\leq i \leq k$ and $[q_{k+1}, \ldots q_n] = [u_{k+1}, \ldots u_n] R$ where $R \in O(n-k)$ is any orthogonal matrix.

In order to simplify the proof, we start by  assuming that all the eigenvalues have multiplicity one, and 
we optimize iteratively over each unit vector $q_i$.  We start by $q_1$: augmenting the Lagrangian with the the constraint $\|q_1\|^2 =1$ using the Lagrange multiplier $\lambda_1$ and differentiating, we obtain:
\[
\nabla_{q_1} (   AUV(Q) + \lambda \|q_1\|^2) = \Sigma q_1 + \lambda_1 q_1 =0. 
\]
This means that $q_1$ is a unit eigenvector of $\Sigma$. Denoting $\pi(1)$ the index of this eigenvector, we have $q_1^* = u_{\pi(1)}$ and the eigenvalue is $-\lambda_1 = \sigma_{\pi(1)}^2$. 
The criterion for this partially optimal value is now
\[
AUV([q_1^*, q_2 \ldots q_n] ) = \sigma_{\pi(1)}^2 + \sum_{i=2}^k i q_i\trp \Sigma q_i + (k+1)  \sum_{i=k+1}^n  q_i\trp \Sigma q_i.
\]
To take into account the orthogonality of the remaining vectors $q_i$ ($i > 1$) with $q_1^*$ in the optimization, we can project all the above quantities along $u_{\pi(1)}$. Optimizing now for $q_2$ under the constraint $\|q_2\|^2=1$, we find that  $q_2$ is a unit eigenvector of $\Sigma - \sigma_{\pi(1)}^2 u_{\pi(1)} u_{\pi(1)}\trp$ associated to a non-zero eigenvalue. Denoting $\pi(2)$ the index of this eigenvector (which is thus different from $\pi(1)$ because it has to be non-zero), we have $q_2^* = u_{\pi(2)}$ and the eigenvalue is $-\lambda_2 = 2 \sigma_{\pi(2)}^2$. 

Iterating the process, we conclude that  $q_i^* = u_{\pi(i)}$ for some permutation $\pi$ of the indices $1, \ldots n$. Moreover, the value of the criterion for that permutation is
\[
AUV([q_1^*, q_2^* \ldots q_n^*] ) =  \sum_{i=q}^k i \sigma_{\pi(i)}^2 + (k+1)  \sum_{i=k+1}^n  \sigma_{\pi(i)}^2.
\]
In order to find the global minimum, we now have to compare the values of this criterion for all the possible permutations.

Assuming that $i<j$, we now show that the permutation of two indices $\pi(i)$ and $\pi(j)$ give a lower (or equal) criterion when $\pi(i) < \pi(j)$. Because eigenvalues are sorted in strictly decreasing order, we have $\sigma_{\pi(i)}^2 > \sigma_{\pi(j)}^2$. Thus,  $(\alpha-1) \sigma_{\pi(i)}^2 > (\alpha-1) \sigma_{\pi(j)}^2$ for any $\alpha \geq 1$ and adding $\sigma_{\pi(i)}^2 + \sigma_{\pi(j)}^2$ on both sides, we get $\alpha \sigma_{\pi(i)}^2 + \sigma_{\pi(j)}^2 > \sigma_{\pi(i)}^2 + \alpha \sigma_{\pi(j)}^2$. 
For the value of $\alpha$, we distinguish there cases:
\begin{itemize}
	\item $i<j\leq k$: we take $\alpha = j/i > 1$. multiplying on both sides by the positive value $i$, we get:
	$i \sigma_{\pi(i)}^2 + j \sigma_{\pi(j)}^2 < i \sigma_{\pi(j)}^2 + j \sigma_{\pi(i)}^2$. The value of the criterion is thus strictly lower if $\pi(i) < \pi(j)$. 
	\item $i \leq k< j$:  we take $\alpha = (k+1)/i > 1$ and we get:
	$i \sigma_{\pi(i)}^2 + (k+1) \sigma_{\pi(j)}^2 < i \sigma_{\pi(j)}^2 + (k+1)  \sigma_{\pi(i)}^2$. Once again, the value of the criterion is thus strictly lower if $\pi(i) < \pi(j)$. 
	\item $k < i<j$: here permuting the indices does not change the criterion since $\sigma_{\pi(i)}^2$ and $\sigma_{\pi(j)}^2$ are both counted with the weight $(k+1)$.  
\end{itemize}
In all cases, the criterion is minimized  by swapping indices in the permutation such that $\pi(i) < \pi(j)$ for $i<j$ and $i<k$. The global minimum is thus achieved for the identity permutation $\pi(i) = i$ for the indices $1 \leq i \leq k$. For the higher indices, any linear combination  of the last $n-k$ eigenvectors of $\Sigma$ gives the same value of the criterion. Taking into account the orthonormality constraints, such a linear combination writes $[q_{k+1}, \ldots q_n] = [u_{k+1}, \ldots u_n] R$ for some orthonormal $(n-k)\times (n-k)$ matrix $R$. 

When some eigenvalues of $\Sigma$ have a multiplicity larger than one, then the corresponding eigenvectors cannot be uniquely determined since they can be rotated within the eigenspace.  With our assumptions, this can only occur within the last $n-k$ eigenvalues and this does not change anyway the value of the criterion.
We have thus proved the following theorem.

\begin{theorem}[Euclidean PCA as an optimization in the flag space] $ $ \\
Let ${\hat Y} = \{ \hat y_i \}_{i=1}^N$ be a set of $N$ data points in $\R^n$. We denote as usual the mean by $\bar y = \frac{1}{N} \sum_{i=1}^N  \hat y_i$ and the empirical covariance matrix by $\Sigma = \frac{1}{N} \sum_{i=1}^N  (\hat y_i -\bar y) (\hat y_i -\bar y)\trp$. Its spectral decomposition is denoted $\Sigma = \sum_{j=1}^n \sigma_j^2 u_j u_j\trp$ with the eigenvalues sorted in decreasing order. We assume that the first $k+1$ eigenvalues have multiplicity one, so that the order from $\sigma_1$ to $\sigma_{k+1}$ is strict.  

Then the partial flag of affine subspaces $Fl(x_0\prec x_1 \ldots  \prec x_k)$  optimizing the AUV criterion:
\[
AUV(Fl(x_0\prec x_1 \ldots  \prec x_k)) = \sum_{i=0}^k \sigma^2_{out}( Fl_i(x_0\prec x_1 \ldots  \prec x_k ) )
\]
is totally ordered  and can be parameterized by $x_0 = \bar y$, $x_i = x_0 + u_i$ for $1 \leq i \leq k$. The parametrization by points is not unique but the flag of subspaces which is generated is and is equal to the flag generated by the PCA  modes up to mode $k$ included.
\end{theorem}

\end{document}